\documentclass[12pt, reqno, a4paper]{amsart}

\usepackage{amssymb}
\usepackage{hyperref}

\theoremstyle{plain}

\newtheorem*{corollary}{Corollary}
\newtheorem{lemma}{Lemma}
\newtheorem{theorem}{Theorem}
\newtheorem*{conjecture}{Conjecture}

\theoremstyle{remark}
\newtheorem*{remark}{Remark}

\theoremstyle{definition}

\newtheorem{example}{Example}

\DeclareMathOperator{\Id}{Id}
\DeclareMathOperator{\id}{id}

\DeclareMathOperator{\tr}{tr}

\DeclareMathOperator{\End}{End}
\DeclareMathOperator{\Aut}{Aut}

\DeclareMathOperator{\im}{im}

\DeclareMathOperator{\sign}{sign}
\DeclareMathOperator{\GL}{GL}
\DeclareMathOperator{\UT}{UT}
\DeclareMathOperator{\Alt}{Alt}

\DeclareMathOperator{\Hom}{Hom}
\DeclareMathOperator{\PIexp}{PIexp}

\begin{document}

\title[Amitsur's conjecture for algebras with a generalized Hopf action]{Amitsur's conjecture for associative algebras
with a generalized Hopf action}

\author{A.\,S.~Gordienko}

\address{Memorial University of Newfoundland, St. John's, NL, Canada}
\email{asgordienko@mun.ca}

\begin{abstract}
We prove the analog of Amitsur's conjecture
on asymptotic behavior for codimensions of several generalizations of polynomial identities
for finite dimensional associative algebras over a field of characteristic $0$,
including $G$-identities
for any finite (not necessarily Abelian) group~$G$ and $H$-identities
 for a finite dimensional semisimple Hopf algebra~$H$.
  In addition, we prove that the Hopf
 PI-exponent of Sweedler's $4$-dimensional algebra with the action of its dual equals $4$.
\end{abstract}

\thanks{
Supported by post doctoral fellowship
from Atlantic Association for Research
in Mathematical Sciences (AARMS), Atlantic Algebra Centre (AAC),
Memorial University of Newfoundland (MUN), and
Natural Sciences and Engineering Research Council of Canada (NSERC)}

\medskip

\maketitle

\section{Introduction}

In the 1980's, a conjecture about the asymptotic behaviour
of codimensions of ordinary polynomial identities was made
by S.A.~Amitsur for algebras over a field of characteristic $0$.
 Amitsur's conjecture was proved in 1999 by
A.~Giambruno and M.V.~Zaicev~\cite[Theorem~6.5.2]{ZaiGia} for associative algebras, in 2002 by M.V.~Zaicev~\cite{ZaiLie}
 for finite dimensional Lie algebras, and in 2011 by A.~Giambruno,
 I.P.~Shestakov, M.V. Zaicev~\cite{GiaSheZai} for finite dimensional Jordan and alternative
 algebras. The author proved its analog
 for polynomial identities of finite dimensional representations of Lie
 algebras~\cite{ASGordienko}.
 
  Alongside with ordinary polynomial
identities of algebras, graded polynomial identities, $G$- and
$H$-identities are
important too~\cite{BahtGiaZai, BahtZaiGraded, BahtZaiGradedExp, BahtZaiSehgal, BahturinLinchenko,
 BereleHopf, Linchenko}.
 Usually, to find such identities is easier
  than to find the ordinary ones. Furthermore, the graded polynomial identities, $G$- and
$H$-identities completely determine the ordinary polynomial identities. 
Therefore the question arises whether the conjecture
holds for graded codimensions, $G$- and $H$-codimensions.
The analog of Amitsur's conjecture
for codimensions of graded identities was proved in 2010--2011 by
E.~Aljadeff,  A.~Giambruno, and D.~La~Mattina~\cite{AljaGia, AljaGiaLa, GiaLa}
  for all associative PI-algebras graded by a finite group.
   As a consequence, they proved the analog  of the conjecture for $G$-codimensions
   for any associative PI-algebra with an action of a finite Abelian group $G$ by automorphisms.
    In 2011 the author~\cite{ASGordienko2} proved 
  the analog of Amitsur's conjecture for graded polynomial identities
  of finite dimensional
Lie algebras graded by a finite Abelian group
and  for $G$-identities
of finite dimensional
Lie algebras with an action of any finite group (not necessarily Abelian).
The case when $G=\mathbb Z_2$
acts on a finite dimensional associative algebra by automorphisms and anti-automorphisms
(i.e. polynomial identities with involution)
 was considered by A.~Giambruno and
 M.V.~Zaicev~\cite[Theorem~10.8.4]{ZaiGia}
 in 1999.

This article is concerned with the analog of Amitsur's conjecture for 
 $G$-codimensions 
(Subsection~\ref{SubsectionGAction}),
where an arbitrary finite group $G$ acts by automorphisms and anti-automorphisms,
 and for $H$-codimensions of $H$-module associative algebras where $H$ is a Hopf algebra
(Subsection~\ref{SubsectionHopf}).
As shown in Subsection~\ref{SubsectionDuality}, the case of graded codimensions is a particular case
of $H$-codimensions. Hence we obtain a new proof of Amitsur's conjecture for graded polynomial identities
of finite dimensional associative algebras graded by a finite group. The case of $H$-codimensions does not always include the case of $G$-codimensions,
namely, when $G$ acts on an algebra not only by automorphisms, but by anti-automorphisms too.
However, in Section~\ref{SectionGeneralizedHopf} we consider the generalized Hopf action that
embraces all three situations.

 In Subsection~\ref{SubsectionGenHopfPIexp} we provide an explicit formula
for the (generalized) Hopf PI-exponent that is a natural generalization of the formula
for the ordinary PI-exponent~\cite[Section 6.2]{ZaiGia}. Of course, this formula can be used for
graded codimensions and $G$-codimensions as well. The formula has immediate applications.
In particular, in Section~\ref{SectionExamples} we apply it 
to calculate  the (generalized) Hopf PI-exponent for several important classes of algebras.
In Subsection~\ref{SubsectionSweedler} we study the asymptotic behaviour of $H$-codimensions
of $4$-dimensional Sweedler's algebra with the action of its dual. In this case we cannot use
theorems from Subsection~\ref{SubsectionHopf} since the Jacobson radical of Sweedler's algebra is not stable
under this action.
However it is possible to apply the techniques developed in Sections~\ref{SectionAlt}--\ref{SectionLower} and prove directly that the Hopf PI-exponent of Sweedler's algebra equals $4$.

The results obtained provide a useful tool to study polynomial identities in associative algebras and hence to study associative algebras themselves. They allow for future applications.

All the codimensions discussed in the article do not change upon an extension of the base field.
The proof is analogous to the case of codimensions of ordinary polynomial identities~\cite[Theorem~4.1.9]{ZaiGia}. Thus without loss of generality we may assume the field to be algebraically closed.
In the main results we require the field to be of characteristic $0$.

\subsection{Polynomial $G$-identities and their codimensions}\label{SubsectionGAction}

We use the exponential notation for the action of a
 group.
Let $A$ be an associative algebra over a field $F$.
Recall that $\psi \in \GL(A)$ is an {\itshape
automorphism} of $A$ if $(ab)^\psi = a^\psi b^\psi$
for all $a,b \in A$
and {\itshape
anti-automorphism} of $A$ if $(ab)^\psi = b^\psi a^\psi$
for all $a, b \in A$. Automorphisms of $A$
form the group denoted by $\Aut(A)$.
Automorphisms and anti-automorphisms of $A$
form the group denoted by $\Aut^{*}(A)$.
Note that $\Aut(A)$ is a normal subgroup of $\Aut^{*}(A)$
of index ${}\leqslant 2$.

Let $G$ be a group
with a fixed (normal) subgroup $G_0$ of index ${}\leqslant 2$.
  We say that an associative algebra $A$ is an \textit{algebra with $G$-action}
  or a \textit{$G$-algebra}
   if $A$ is endowed with a homomorphism $\varphi \colon G \to
  \Aut^{*}(A)$ such that $\varphi^{-1}(\Aut(A))=G_0$.
 Denote by $F\langle X | G \rangle$
the free associative algebra over $F$ with free formal generators $x^g_j$, $j\in\mathbb N$,
 $g \in G$. Here $X := \lbrace x_1, x_2, x_3, \ldots \rbrace$, $x_j := x_j^1$. Define
 $$(x_{i_1}^{g_1} x_{i_2}^{g_2}\ldots x_{i_{n-1}}^{g_{n-1}}
  x_{i_n}^{g_n})^h :=
x_{i_1}^{hg_1} x_{i_2}^{hg_2}\ldots x_{i_{n-1}}^{hg_{n-1}} x_{i_n}^{hg_n}
\text { for } h \in G_0, $$
 $$(x_{i_1}^{g_1} x_{i_2}^{g_2}\ldots x_{i_{n-1}}^{g_{n-1}}
  x_{i_n}^{g_n})^h :=
x_{i_n}^{hg_n} x_{i_{n-1}}^{hg_{n-1}}  \ldots x_{i_2}^{hg_2}x_{i_1}^{hg_1}
\text { for } h \in G\backslash G_0.$$
 Then $F\langle X | G \rangle$ becomes the free $G$-algebra with
 free generators $x_j$, $j \in \mathbb N$. We call its elements
 $G$-polynomials.
 Let $A$ be an associative  $G$-algebra over $F$. A $G$-polynomial
 $f(x_1, \ldots, x_n)\in F\langle X | G \rangle$
 is a \textit{$G$-identity} of $A$ if $f(a_1, \ldots, a_n)=0$
for all $a_i \in A$. In this case we write
$f \equiv 0$.
The set $\Id^{G}(A)$ of all $G$-identities
of $A$ is an ideal in $F\langle X | G \rangle$ invariant under $G$-action.
If $G=\lbrace e \rbrace$ is the trivial group, then we have the case of ordinary
polynomial identities.

\begin{example}\label{ExampleIdG} Let $M_2(F)$ be the algebra
of $2\times 2$ matrices. Consider $\psi \in \Aut(M_2(F))$
defined by the formula $$\left(
\begin{array}{cc}
a & b \\
c & d
\end{array}
 \right)^\psi := \left(
\begin{array}{rr}
a & -b \\
-c & d
\end{array}
 \right).$$
Then $[x+x^{\psi},y+y^{\psi}]\in \Id^{G}(M_2(F))$
where
$G=\langle \psi \rangle \cong \mathbb Z_2$.
Here $[x,y]:=xy-yx$.
\end{example}

\begin{example}\label{ExampleIdG2} Consider $\psi \in \Aut^*(M_2(F))$
defined by the formula $$\left(
\begin{array}{cc}
a & b \\
c & d
\end{array}
 \right)^\psi := \left(
\begin{array}{rr}
a & c \\
b & d
\end{array}
 \right),$$
 i.e. $\psi$ is the transposition.
Then $[x-x^{\psi},y-y^{\psi}]\in \Id^{G}(M_2(F))$
where
$G=\langle \psi \rangle \cong \mathbb Z_2$.
\end{example}

Denote by $P^G_n$ the space of all multilinear $G$-polynomials
in $x_1, \ldots, x_n$, $n\in\mathbb N$, i.e.
$$P^{G}_n = \langle x^{g_1}_{\sigma(1)}
x^{g_2}_{\sigma(2)}\ldots x^{g_n}_{\sigma(n)}
\mid g_i \in G, \sigma\in S_n \rangle_F \subset F \langle X | G \rangle$$
where $S_n$ is the $n$th symmetric group.
Then the number $c^G_n(A):=\dim\left(\frac{P^G_n}{P^G_n \cap \Id^G(A)}\right)$
is called the $n$th \textit{codimension of polynomial $G$-identities}
or the $n$th \textit{$G$-codimension} of $A$.

 Let $c_n(A)$ be the $n$th ordinary codimension, which equals
 the $n$th $G$-codimension for $G=\lbrace e \rbrace$.
 Then by~\cite[Lemmas 10.1.2 and 10.1.3]{ZaiGia} we have $c_n(A) \leqslant c^{G}_n(A)
  \leqslant |G|^n c_n(A)$ for all  $n \in \mathbb N$.

The analog of Amitsur's conjecture for $G$-codimensions can be formulated
as follows.

\begin{conjecture} There exists
 $\PIexp^G(A):=\lim\limits_{n\to\infty}
 \sqrt[n]{c^G_n(A)} \in \mathbb Z_+$.
\end{conjecture}

\begin{theorem}\label{TheoremMainG}
Let $A$ be a finite dimensional non-nilpotent associative algebra
over a field $F$ of characteristic $0$. Suppose a finite not necessarily
Abelian group $G$ acts on $A$ by automorphisms and anti-automorphisms.
 Then there exist constants $C_1, C_2 > 0$, $r_1, r_2 \in \mathbb R$, $d \in \mathbb N$ such that $C_1 n^{r_1} d^n \leqslant c^{G}_n(A) \leqslant C_2 n^{r_2} d^n$ for all $n \in \mathbb N$.
\end{theorem}

\begin{corollary}
The above analog of Amitsur's conjecture holds
 for such codimensions.
\end{corollary}

\begin{remark}
If $A$ is nilpotent, i.e. $x_1 \ldots x_p\equiv 0$ for some $p\in\mathbb N$, then
$P^{G}_n \subseteq \Id^{G}(A)$ and $c^G_n(A)=0$ for all $n \geqslant p$.
\end{remark}

Theorem~\ref{TheoremMainG} is obtained as a consequence of Theorem~\ref{TheoremMain}
in Subsection~\ref{SubsectionProofofTheoremMainG}.

\subsection{$H$-identities and their codimensions}\label{SubsectionHopf}

Analogously, one can consider polynomial $H$-identities
of $H$-module algebras where $H$ is a Hopf algebra. An algebra $A$
over a field $F$
is an \textit{$H$-module algebra}
or an \textit{algebra with an $H$-action},
if $A$ is endowed with a homomorphism $H \to \End_F(A)$ such that
$h(ab)=(h_{(1)}a)(h_{(2)}b)$
for all $h \in H$, $a,b \in A$. Here we use Sweedler's notation
$\Delta h = h_{(1)} \otimes h_{(2)}$ where $\Delta$ is the comultiplication
in $H$.
We refer the reader to~\cite{Danara, Montgomery, Sweedler}
   for an account
  of Hopf algebras and algebras with Hopf algebra actions.

Let $H$ be a Hopf algebra with a basis $(\gamma_\beta)_{\beta \in \Lambda}$.
Denote by $F \langle X | H \rangle$
the free associative algebra over $F$ with free formal
 generators $x_i^{\gamma_\beta}$, $\beta \in \Lambda$, $i \in \mathbb N$.
 Let $x_i^h := \sum_{\beta \in \Lambda} \alpha_\beta x_i^{\gamma_\beta}$
 for $h= \sum_{\beta \in \Lambda} \alpha_\beta \gamma_\beta$, $\alpha_\beta \in F$,
 where only finite number of $\alpha_\beta$ are nonzero.
Here $X := \lbrace x_1, x_2, x_3, \ldots \rbrace$, $x_j := x_j^1$, $1 \in H$. Define
 $$h \cdot (x_{i_1}^{\gamma_{\beta_1}} x_{i_2}^{\gamma_{\beta_2}}\ldots x_{i_{n-1}}^{\gamma_{\beta_{n-1}}}
  x_{i_n}^{\gamma_{\beta_n}}) :=
x_{i_1}^{h_{(1)}\gamma_{\beta_1}} x_{i_2}^{h_{(2)}\gamma_{\beta_2}}\ldots x_{i_{n-1}}^{h_{(n-1)}\gamma_{\beta_{n-1}}} x_{i_n}^{h_{(n)}\gamma_{\beta_n}}$$
for $h \in H$,  $\beta_1, \beta_2, \ldots, \beta_m \in \Lambda$,
where $h_{(1)}\otimes h_{(2)} \otimes \ldots \otimes h_{(n)}$
is the image of $h$ under the comultiplication $\Delta$
applied $(n-1)$ times.  Then $F\langle X | H \rangle$ becomes the free $H$-module algebra with
 free generators $x_j$, $j \in \mathbb N$. We call its elements
 $H$-polynomials.
 Let $A$ be an associative  $H$-module algebra over $F$. An $H$-polynomial
 $f(x_1, \ldots, x_n)\in F\langle X | H \rangle$
 is an \textit{$H$-identity} of $A$ if $f(a_1, \ldots, a_n)=0$
for all $a_i \in A$. In other words, $f$ is an $H$-identity of $A$
if and only if $\psi(f)=0$ for any $H$-homomorphism $\psi \colon
F\langle X | H \rangle \to A$. In this case we write
$f \equiv 0$.
The set $\Id^{H}(A)$ of all $H$-identities
of $A$ is an ideal in $F\langle X | H \rangle$ invariant under the $H$-action.
If $H=F$, then we have the case of ordinary
polynomial identities.

\begin{example}\label{ExampleIdFG}
Let $A$ be an associative algebra with an action 
of a group $G$ by automorphisms only. Note that $H=FG$
is a Hopf algebra with $\Delta g = g \otimes g$, $Sg = g^{-1}$,
$\varepsilon(g)=1$,
for all $g \in G$. Thus group action becomes a Hopf action,
and we may identify $F\langle X | H \rangle=F\langle X | G \rangle$.
Furthermore, $\Id^H(A)=\Id^G(A)$.
\end{example}

\begin{example}\label{ExampleIdH}
Let $M_2(F)$ be the algebra
of $2\times 2$ matrices. Consider $e_0, e_1 \in \End_F(M_2(F))$
defined by the formulas $$e_0 \left(
\begin{array}{cc}
a & b \\
c & d
\end{array}
 \right) := \left(
\begin{array}{rr}
a & 0 \\
0 & d
\end{array}
 \right)$$
 and $$e_1 \left(
\begin{array}{cc}
a & b \\
c & d
\end{array}
 \right) := \left(
\begin{array}{rr}
0 & b \\
c & 0
\end{array}
 \right).$$
 Then $H:=F e_0 \oplus F e_1$ (direct sum of ideals)
 is a Hopf algebra with the counit $\varepsilon$, where $\varepsilon(e_0):=1$,
 $\varepsilon(e_1):=0$, the comultiplication $\Delta$ where
 $$\Delta(e_0):=e_0 \otimes e_0 + e_1 \otimes e_1,$$
 $$\Delta(e_1):=e_0 \otimes e_1 + e_1 \otimes e_0,$$ and the antipode $S:=\id$.
 Note that $x^{e_0}y^{e_0}-y^{e_0}x^{e_0}\in \Id^{H}(M_2(F))$.
\end{example}

Denote by $P^H_n$ the space of all multilinear $H$-polynomials
in $x_1, \ldots, x_n$, $n\in\mathbb N$, i.e.
$$P^{H}_n = \langle x^{h_1}_{\sigma(1)}
x^{h_2}_{\sigma(2)}\ldots x^{h_n}_{\sigma(n)}
\mid h_i \in H, \sigma\in S_n \rangle_F \subset F \langle X | H \rangle.$$
Then the number $c^H_n(A):=\dim\left(\frac{P^H_n}{P^H_n \cap \Id^H(A)}\right)$
is called the $n$th \textit{codimension of polynomial $H$-identities}
or the $n$th \textit{$H$-codimension} of $A$.

Note that in Example~\ref{ExampleIdFG} we have $c^H_n(A)=c^G_n(A)$.

The analog of Amitsur's conjecture for $H$-codimensions can be formulated
as follows.

\begin{conjecture} There exists
 $\PIexp^H(A):=\lim\limits_{n\to\infty}
 \sqrt[n]{c^H_n(A)} \in \mathbb Z_+$.
\end{conjecture}

We call $\PIexp^H(A)$ the \textit{Hopf PI-exponent} of $A$.

\begin{theorem}\label{TheoremMainH}
Let $A$ be a finite dimensional non-nilpotent associative algebra
over an algebraically closed field $F$ of characteristic $0$. Suppose a finite dimensional
Hopf algebra $H$ acts on $A$ in such a way that the Jacobson radical $J:=J(A)$
is $H$-invariant and $A = B \oplus J$
(direct sum of $H$-submodules) where $B=B_1 \oplus \ldots \oplus B_q$ (direct sum of $H$-invariant ideals),
$B_i$ are $H$-simple semisimple algebras.
 Then there exist constants $C_1, C_2 > 0$, $r_1, r_2 \in \mathbb R$, $d \in \mathbb N$ such that $C_1 n^{r_1} d^n \leqslant c^{H}_n(A) \leqslant C_2 n^{r_2} d^n$ for all $n \in \mathbb N$.
\end{theorem}

\begin{remark}
If $A$ is nilpotent, i.e. $x_1 \ldots x_p\equiv 0$ for some $p\in\mathbb N$, then
$P^{H}_n \subseteq \Id^{H}(A)$ and $c^H_n(A)=0$ for all $n \geqslant p$.
\end{remark}

Theorem~\ref{TheoremMainH} will be obtained as a consequence of Theorem~\ref{TheoremMain}
in Subsection~\ref{SubsectionGenHopfPIexp}.
Note that here we require the existence of $H$-invariant Wedderburn
decompositions. However, if $H$ is semisimple, then such decompositions
always exist. We discuss this in Section~\ref{SectionHSemisimple} and derive from Theorem~\ref{TheoremMainH}
the following

\begin{theorem}\label{TheoremMainHSS}
Let $A$ be a finite dimensional non-nilpotent $H$-module associative algebra
over a field $F$ of characteristic $0$, where $H$ is a finite dimensional
semisimple Hopf algebra.
 Then there exist constants $C_1, C_2 > 0$, $r_1, r_2 \in \mathbb R$,
  $d \in \mathbb N$ such that $C_1 n^{r_1} d^n \leqslant c^{H}_n(A)
   \leqslant C_2 n^{r_2} d^n$ for all $n \in \mathbb N$.
\end{theorem}

\begin{corollary}
The above analog of Amitsur's conjecture holds
 for such codimensions.
\end{corollary}
%
%\begin{example}\label{ExampleGroupTwisted}
%If $G$ is a finite group, then $H:=FG$ is a finite dimensional
%Hopf algebra. Using twisting we can construct new   finite dimensional semisimple
%Hopf algebras. Namely, if $\gamma \in H \otimes H$ is an invertible element
%such that $(1 \otimes \gamma)(\id_H \otimes \Delta)(\gamma)
%=(\gamma \otimes 1)(\Delta \otimes \id_H)(\gamma)$
%and $ (\id_H \otimes \varepsilon)(\gamma) = (\varepsilon \otimes \id_H )(\gamma) = 1_H$,
%then $H$ with the same multiplication and the comultiplication $\Delta_\gamma(h)
%= \gamma \Delta(h) \gamma^{-1}$ for $h \in H$, is again a finite dimensional semisimple Hopf algebra. 
%\end{example}

\subsection{Graded polynomial identities and their codimensions}\label{SubsectionGraded}

Let $G$ be a group and $F$ be a field. Denote by $F\langle X^{\mathrm{gr}} \rangle $ the free $G$-graded associative  algebra over $F$ on the countable set $$X^{\mathrm{gr}}:=\bigcup_{g \in G}X^{(g)},$$ $X^{(g)} = \{ x^{(g)}_1,
x^{(g)}_2, \ldots \}$,  i.e. the algebra of polynomials
 in non-commuting variables from $X^{\mathrm{gr}}$.
  The indeterminates from $X^{(g)}$ are said to be homogeneous of degree
$g$. The $G$-degree of a monomial $x^{(g_1)}_{i_1} \dots x^{(g_t)}_{i_t} \in F\langle
 X^{\mathrm{gr}} \rangle $ is defined to
be $g_1 g_2 \dots g_t$, as opposed to its total degree, which is defined to be $t$. Denote by
$F\langle
 X^{\mathrm{gr}} \rangle^{(g)}$ the subspace of the algebra $F\langle
 X^{\mathrm{gr}} \rangle$ spanned
 by all the monomials having
$G$-degree $g$. Notice that $$F\langle
 X^{\mathrm{gr}} \rangle^{(g)} F\langle
 X^{\mathrm{gr}} \rangle^{(h)} \subseteq F\langle
 X^{\mathrm{gr}} \rangle^{(gh)},$$ for every $g, h \in G$. It follows that
$$F\langle
 X^{\mathrm{gr}} \rangle =\bigoplus_{g\in G} F\langle
 X^{\mathrm{gr}} \rangle^{(g)}$$ is a $G$-grading.
  Let $f=f(x^{(g_1)}_{i_1}, \dots, x^{(g_t)}_{i_t}) \in F\langle
 X^{\mathrm{gr}} \rangle$.
We say that $f$ is
a \textit{graded polynomial identity} of
 a $G$-graded algebra $A=\bigoplus_{g\in G}
A^{(g)}$
and write $f\equiv 0$
if $f(a^{(g_1)}_{i_1}, \dots, a^{(g_t)}_{i_t})=0$
for all $a^{(g_j)}_{i_j} \in A^{(g_j)}$, $1 \leqslant j \leqslant t$.
  The set $\Id^{\mathrm{gr}}(A)$ of graded polynomial identities of
   $A$ is
a graded ideal of $F\langle
 X^{\mathrm{gr}} \rangle$.
The case of ordinary polynomial identities is included
for the trivial group $G=\lbrace e \rbrace$.

\begin{example}\label{ExampleIdGr}
 Let $G=\mathbb Z_2 = \lbrace \bar 0, \bar 1 \rbrace$,
$M_2(F)=M_2(F)^{(\bar 0)}\oplus M_2(F)^{(\bar 1)}$
where $M_2(F)^{(\bar 0)}=\left(
\begin{array}{cc}
F & 0 \\
0 & F
\end{array}
 \right)$ and $M_2(F)^{(\bar 1)}=\left(
\begin{array}{cc}
0 & F \\
F & 0
\end{array}
 \right)$. Then  $x^{(\bar 0)} y^{(\bar 0)} - y^{(\bar 0)} x^{(\bar 0)}
\in \Id^{\mathrm{gr}}(M_2(F))$.
\end{example}

Let
$P^{\mathrm{gr}}_n := \langle x^{(g_1)}_{\sigma(1)}
x^{(g_2)}_{\sigma(2)}\ldots x^{(g_n)}_{\sigma(n)}
\mid g_i \in G, \sigma\in S_n \rangle_F \subset F \langle X^{\mathrm{gr}} \rangle$, $n \in \mathbb N$.
Then the number $$c^{\mathrm{gr}}_n(A):=\dim\left(\frac{P^{\mathrm{gr}}_n}{P^{\mathrm{gr}}_n \cap \Id^{\mathrm{gr}}(A)}\right)$$
is called the $n$th \textit{codimension of graded polynomial identities}
or the $n$th \textit{graded codimension} of $A$.

The analog of Amitsur's conjecture for graded codimensions can be formulated
as follows.

\begin{conjecture} There exists
 $\PIexp^{\mathrm{gr}}(A):=\lim\limits_{n\to\infty} \sqrt[n]{c^\mathrm{gr}_n(A)} \in \mathbb Z_+$.
\end{conjecture}

Using techniques different from ours, E.~Aljadeff and A.~Giambruno~\cite{AljaGia} proved
in 2011 the analog Amitsur's conjecture for graded codimensions of all associative (not necessarily finite dimensional) PI-algebras.
However, for finite dimensional algebras,
this result can be easily derived from Theorem~\ref{TheoremMainHSS}
 using Lemma~\ref{LemmaGradAction} in Subsection~\ref{SubsectionDuality} below.
\begin{theorem}\label{TheoremMainGr}
Let $A$ be a finite dimensional non-nilpotent associative algebra
over a field $F$ of characteristic $0$, graded by a finite group $G$. Then
there exist constants $C_1, C_2 > 0$, $r_1, r_2 \in \mathbb R$, $d \in \mathbb N$
such that $C_1 n^{r_1} d^n \leqslant c^{\mathrm{gr}}_n(A) \leqslant C_2 n^{r_2} d^n$
for all $n \in \mathbb N$.
\end{theorem}
\begin{corollary}
The above analog of Amitsur's conjecture holds for such codimensions.
\end{corollary}
\begin{remark}
If $A$ is nilpotent, i.e. $x_1 \ldots x_p \equiv 0$ for some $p\in\mathbb N$, then  $P^{\mathrm{gr}}_n \subseteq \Id^{\mathrm{gr}}(A)$ and $c^{\mathrm{gr}}_n(A)=0$ for all $n \geqslant p$.
\end{remark}

\subsection{Duality between gradings and Hopf actions}\label{SubsectionDuality}

Let $A=\bigoplus_{g\in G} A^{(g)}$ be a graded algebra over a field $F$.
 Then $A$ becomes an $FG$-comodule algebra
where $FG$ is the group algebra. The comodule map $\rho \colon A \to A \otimes FG$
is defined by $\rho(a_g): =a_g \otimes g$ for $a_g \in A^{(g)}$.
Again we use Sweedler's notation $\rho(a)=a_{(0)}\otimes a_{(1)}$, $a_{(0)} \in A$,
$a_{(1)} \in FG$. 

 Suppose $G$
is finite. Let $H:=(FG)^*$ be the Hopf algebra dual to $FG$. Then $A$ is an $H$-module algebra
  where $h a = h(a_{(1)}) a_{(0)}$, $h \in H$, $a \in A$.
 Conversely, each $H$-module algebra $A$ has the following $G$-grading:
 $A=\bigoplus_{g\in G} A^{(g)}$ where $$A^{(g)}=\lbrace a \in A
 \mid h  a  = h(g) a \text{ for all } h \in H\rbrace.$$
  In particular, $F\langle X | H \rangle$ and $F\langle X^{\mathrm{gr}} \rangle$
 are both $H$-module and $G$-graded algebras.
 Let $(h_g)_{g\in G}$ be the basis in $H=(FG)^*$ dual
 to the basis $(g)_{g\in G}$ of $FG$, i.e. $$h_{g_1}(g_2)=\left\lbrace
  \begin{array}{ll}
  1, & g_1 = g_2, \\
  0, & g_1 \ne g_2.
  \end{array} \right.$$
  Note that $$h_{g_1}h_{g_2}=\left\lbrace
  \begin{array}{ll}
  h_{g_1}, & g_1 = g_2, \\
  0, & g_1 \ne g_2,
  \end{array} \right.$$ i.e. $H$ is the direct sum of fields.
 Moreover, the $H$-homomorphism $\varphi \colon F\langle X | H \rangle
 \to F\langle X^{\mathrm{gr}} \rangle$ defined
 by $\varphi(x_j)=\sum_{g\in G} x^{(g)}_j$
 is an isomorphism since $\varphi^{-1}$
 is the graded homomorphism $F\langle X^{\mathrm{gr}} \rangle
 \to F\langle X | H \rangle$ defined by $\varphi^{-1}(x^{(g)}_j)
 = x^{h_g}_j$, $g\in G$, $j \in \mathbb N$.

\begin{lemma}\label{LemmaGradAction}
Let $A$ be a $G$-graded associative algebra where $G$ is a finite group.
Consider the corresponding $H$-action on $A$ where $H=(FG)^*$. Then
 \begin{enumerate}
\item $\varphi(\Id^H(A))=\Id^{\mathrm{gr}}(A)$;
\item $c^{H}_n(A)=c^{\mathrm{gr}}_n(A)$.
\end{enumerate}
\end{lemma}
\begin{proof}
The first assertion is evident. The second assertion follows
from the first one and the equality
$\varphi(P^H_n)=P^{\mathrm{gr}}_n$.
\end{proof}

 \begin{remark}
  The $H$-action and the polynomial $H$-identity from Example~\ref{ExampleIdH}
 are dual to the grading and the graded polynomial identity from Example~\ref{ExampleIdGr}.
 \end{remark}

Using Lemma~\ref{LemmaGradAction} and the fact that $(FG)^*$ is the sum of fields, we deduce
 Theorem~\ref{TheoremMainGr} from Theorem~\ref{TheoremMainHSS}.
 
  \section{$H$-module algebras for semisimple $H$}\label{SectionHSemisimple}
  
  In this section we derive Theorem~\ref{TheoremMainHSS} from Theorem~\ref{TheoremMainH}.
  
  Recall that $t \in H$ is  a \textit{left integral}
if $ht=\varepsilon(h)t$ for all $H$.

\begin{lemma}\label{LemmaIntegral}
Let $t$ be a left integral of a finite dimensional semisimple Hopf algebra
over a field $F$ of characteristic~$0$. 
Then $t_{(1)}St_{(3)}\otimes t_{(2)}=1_H \otimes t$.
\end{lemma}  
\begin{proof}
By~\cite[Exercise~7.4.7]{Danara},
$t_{(1)}\otimes t_{(2)}
= t_{(2)}\otimes t_{(1)}$. Applying $\Delta \otimes \id_H$,
we obtain
$$t_{(1)}\otimes t_{(2)} \otimes t_{(3)}
= t_{(2)}\otimes t_{(3)} \otimes t_{(1)}.$$
Interchanging the last two multipliers,
we get
$$
t_{(1)}\otimes t_{(3)} \otimes t_{(2)}
= t_{(2)}\otimes t_{(1)} \otimes t_{(3)}.
$$
Thus $$t_{(1)}St_{(3)}\otimes t_{(2)} = t_{(2)}St_{(1)}\otimes t_{(3)}
=\varepsilon(t_{(1)})\cdot 1_H \otimes t_{(2)}=1_H \otimes t$$
since $t_{(2)}St_{(1)}=S(t_{(1)}St_{(2)})=S(\varepsilon(t)\cdot 1_H)$.
Here we use that $S^2=\id_H$ by the Larson~--- Radford theorem (see e.g \cite[Theorem~7.4.6]{Danara}).
\end{proof}

Now we derive the $H$-invariant Wedderburn theorem
from the original one.
 
\begin{lemma}\label{LemmaHBSS}
If $H$ is a finite dimensional semisimple Hopf algebra over a field of characteristic~$0$ and $B$ is a finite dimensional
semisimple $H$-module associative algebra, then $B=B_1 \oplus B_2 \oplus \ldots 
\oplus B_q$ (direct sum of ideals) for some $H$-simple subalgebras $B_i$.
\end{lemma}
\begin{proof}
Suppose $I_1$ is an $H$-invariant two-sided ideal of $B$. Since $B$ is a finite dimensional semisimple algebra, then by the Wedderburn theorem
it equals the sum of simple ideals, and we can find a complementary
to $I_1$ two-sided ideal $I_2$, $B = I_1 \oplus I_2$.  Denote by $\pi$ the projection
of $B$ on $I_1$ along $I_2$. Then
$\pi(ab)=a\pi(b)=\pi(a)b$ for all $a,b \in B$.

  Since $H$ is semisimple, by~\cite[Theorem 2.2.1]{Montgomery},
there exists a left integral $t \in H$ with $\varepsilon(t)=1$.
Now we use the generalization of Maschke's trick to Hopf algebras:
let $\tilde\pi(a):=t_{(1)} \pi\bigl((St_{(2)})a\bigr)$.
First, $\im\tilde\pi \subseteq I_1$ since $I_1$ is $H$-invariant.
In addition, if $a \in I_1$, then $$\tilde\pi(a)=
t_{(1)} \pi\bigl((St_{(2)})a\bigr)=t_{(1)} (St_{(2)}) a = \varepsilon(t) a = a.$$
Thus $\tilde\pi$ is projection on $I_1$.
Moreover, for all $a,b \in B$, we have $$\tilde\pi(ab)=t_{(1)} \pi\bigl((St_{(2)})(ab)\bigr)=
t_{(1)} \pi\bigl(((St_{(2)})_{(1)}a)((St_{(2)})_{(2)}b)\bigr)
= $$ \begin{equation}
\label{EqHWedder1}
 t_{(1)} \pi\bigl( ((S({t_{(2)}}_{(2)}))a) ((S({t_{(2)}}_{(1)}))b) \bigr)
= t_{(1)} \pi\bigl( ((St_{(3)})a) ((St_{(2)})b) \bigr)
\end{equation}
since $\Delta S =  \tau (S\otimes S) \Delta$ where $\tau (u \otimes v) = v \otimes u$.
However $$ t_{(1)} \pi\Bigl(((St_{(3)})a)((St_{(2)})b)\Bigr) =
t_{(1)} \Bigl(\pi((St_{(3)})a)\Bigl((St_{(2)})b\Bigr)\Bigr) =$$\begin{equation}
\label{EqHWedder2} 
\Bigl(t_{(1)}\pi\bigl((St_{(4)})a\bigr)\Bigr)\Bigl(t_{(2)}(St_{(3)})b\Bigr)=
t_{(1)} \pi((St_{(2)})a)b=\tilde \pi(a) b\end{equation} and
$$t_{(1)} \pi\Bigl(((St_{(3)})a)((St_{(2)})b)\Bigr) = $$\begin{equation}
\label{EqHWedder3}
t_{(1)} \Bigl(((St_{(3)})a) \pi\bigl((St_{(2)})b\bigr)\Bigr) =
\bigl(t_{(1)}(St_{(4)})a\bigr)  t_{(2)} \pi((St_{(3)})b).\end{equation}
By Lemma~\ref{LemmaIntegral},
$t_{(1)}St_{(4)}\otimes t_{(2)} \otimes t_{(3)}=1_H \otimes t_{(1)} \otimes t_{(2)}$.
Hence \begin{equation}
\label{EqHWedder4}\bigl(t_{(1)}(St_{(4)})a\bigr)  t_{(2)} \pi((St_{(3)})(b))
= a\tilde \pi(b). \end{equation}
Equations~(\ref{EqHWedder1})--(\ref{EqHWedder4})
imply $\tilde \pi(ab)=a\tilde \pi(b)=\tilde \pi(a)b$ for all $a,b \in B$.
Moreover, repeating verbatim
the argument from the proof of \cite[Theorem 2.2.1]{Montgomery},
we obtain $\tilde \pi(ha)=h\tilde \pi(a)$
for all $a \in B$, $h \in H$.
Hence $\ker \tilde \pi$ is a two-sided $H$-invariant ideal of $B$.
Furthermore, $B = (\ker \tilde \pi)\oplus I_1$.
Thus we have proved the splitting property which implies the lemma.
\end{proof}

\begin{proof}[Proof of Theorem~\ref{TheoremMainHSS}.]
Since codimensions do not change upon an extension of the base field,
we may assume $F$ to be algebraically closed.
By~\cite[Theorem~3.8]{LinMontSmall}, the Jacobson radical is $H$-invariant
for any finite dimensional semisimple Hopf algebra $H$. By \cite[Corollary~2.7]{SteVanOyst}, we have
an $H$-invariant Wedderburn~--- Malcev decomposition. Together with
Lemma~\ref{LemmaHBSS} and Theorem~\ref{TheoremMainH} this yields the theorem.
\end{proof}

\section{Generalized Hopf action and $S_n$-cocharacters}\label{SectionGeneralizedHopf}

In order to embrace the case when a group acts by anti-automorphisms as well as automorphisms, we
consider the following generalized $H$-action~\cite[Section~3]{BereleHopf}.

\subsection{Definitions}\label{SubsectionGenHopfDef}

Let $H$ be an associative algebra with $1$.
We say that an associative algebra $A$ is an algebra with a \textit{generalized $H$-action}
if $A$ is endowed with a homomorphism $H \to \End_F(A)$
and for every $h \in H$ there exist $h'_i, h''_i, h'''_i, h''''_i \in H$
such that 
\begin{equation}\label{EqGeneralizedHopf}
h(ab)=\sum_i\bigl((h'_i a)(h''_i b) + (h'''_i b)(h''''_i a)\bigr) \text{ for all } a,b \in A.
\end{equation}

As in Subsection~\ref{SubsectionHopf}, we choose a basis $(\gamma_\beta)_{\beta \in \Lambda}$ in $H$ and
denote by $F \langle X | H \rangle$ 
the free associative algebra over $F$ with free formal
 generators $x_i^{\gamma_\beta}$, $\beta \in \Lambda$, $i \in \mathbb N$.
 Let $x_i^h := \sum_{\beta \in \Lambda} \alpha_\beta x_i^{\gamma_\beta}$
 for $h= \sum_{\beta \in \Lambda} \alpha_\beta \gamma_\beta$, $\alpha_\beta \in F$,
 where only finite number of $\alpha_\beta$ are nonzero.
Here $X := \lbrace x_1, x_2, x_3, \ldots \rbrace$, $x_j := x_j^1$, $1 \in H$.
 We refer to the elements
 of $F\langle X | H \rangle$ as $H$-polynomials.
Note that here we do not consider any $H$-action on $F \langle X | H \rangle$.

Let $A$ be an associative algebra with a generalized $H$-action.
Any map $\psi \colon X \to A$ has a unique homomorphic extension $\bar\psi
\colon F \langle X | H \rangle \to A$ such that $\bar\psi(x_i^h)=h\psi(x_i)$
for all $i \in \mathbb N$ and $h \in H$.
 An $H$-polynomial
 $f \in F\langle X | H \rangle$
 is an \textit{$H$-identity} of $A$ if $\bar\psi(f)=0$
for all maps $\psi \colon X \to A$. In other words, $f(x_1, x_2, \ldots, x_n)$
 is an $H$-identity of $A$
if and only if $f(a_1, a_2, \ldots, a_n)=0$ for any $a_i \in A$.
 In this case we write $f \equiv 0$.
The set $\Id^{H}(A)$ of all $H$-identities
of $A$ is an ideal of $F\langle X | H \rangle$.
Note that our definition of $F\langle X | H \rangle$
depends on the choice of the basis $(\gamma_\beta)_{\beta \in \Lambda}$ in $H$.
However such algebras can be identified in the natural way,
and $\Id^{H}(A)$ is the same.

As in Subsection~\ref{SubsectionHopf}, we denote by $P^H_n$ the space of all multilinear $H$-polynomials
in $x_1, \ldots, x_n$, $n\in\mathbb N$, i.e.
$$P^{H}_n = \langle x^{h_1}_{\sigma(1)}
x^{h_2}_{\sigma(2)}\ldots x^{h_n}_{\sigma(n)}
\mid h_i \in H, \sigma\in S_n \rangle_F \subset F \langle X | H \rangle.$$
Then the number $c^H_n(A):=\dim\left(\frac{P^H_n}{P^H_n \cap \Id^H(A)}\right)$
is called the $n$th \textit{codimension of polynomial $H$-identities}
or the $n$th \textit{$H$-codimension} of $A$.

\begin{example}\label{ExampleIdFG*}
Let $A$ be an associative algebra with an action 
of a group $G$ by automorphisms and anti-automorphisms. Then there
is a generalized $H$-action on $A$ where $H=FG$.
Hence we may identify $F\langle X | H \rangle=F\langle X | G \rangle$.
Furthermore, $\Id^H(A)=\Id^G(A)$ and $c^H_n(A)=c^G_n(A)$.
\end{example}

\subsection{Some bounds for codimensions}\label{SubsectionGenHopfBounds}

As in the case of ordinary codimensions,
 we have the following upper bound:
 
 \begin{lemma}\label{LemmaCodimDim}
Let $A$ be a finite dimensional algebra with a generalized $H$-action
over any field~$F$ and let~$H$ be any associative algebra with $1$. Then
$c_n^H(A) \leqslant (\dim A)^{n+1}$ for all  $n \in \mathbb N$.
\end{lemma}
\begin{proof}
Consider $H$-polynomials as $n$-linear maps from $A$ to $A$.
Then we have a natural map $P^{H}_n \to \Hom_{F}(A^{{}\otimes n}; A)$
with the kernel $P^{H}_n \cap \Id^H(A)$
that leads to the embedding $$\frac{P^{H}_n}{P^{H}_n \cap \Id^H(A)}
\hookrightarrow \Hom_{F}(A^{{}\otimes n}; A).$$
Thus $$c^H_n(A)=\dim \left(\frac{P^{H}_n}{P^{H}_n \cap \Id^H(A)}\right)
\leqslant \dim \Hom_{F}(A^{{}\otimes n}; A)=(\dim A)^{n+1}.$$
\end{proof}
\begin{corollary}
Let $A$ be a finite dimensional algebra over any field $F$ with a $G$-action
by automorphisms and anti-automorphisms
 and $G$ be any group. Then
$c_n^G(A) \leqslant (\dim A)^{n+1}$ for all  $n \in \mathbb N$.
\end{corollary}
\begin{corollary}
Let $A$ be a finite dimensional algebra over any field $F$ graded by a finite 
group~$G$. Then
$c^{\mathrm{gr}}_n(A) \leqslant (\dim A)^{n+1}$ for all  $n \in \mathbb N$.
\end{corollary}
\begin{proof} We apply Lemma~\ref{LemmaGradAction}. \end{proof}

Denote by $P_n$ the space of ordinary multilinear polynomials
in the noncommuting variables $x_1, \ldots, x_n$ and by $\Id(A)$
the set of ordinary polynomial identities of $A$.
In other words, $P_n = P^{\lbrace e \rbrace}_n$ and $\Id(A)=\Id^{\lbrace e \rbrace}(A)$
where $\lbrace e \rbrace$ is the trivial group acting on $A$. Then $c_n(A)
:=\dim\frac{P_n}{P_n \cap \Id(A)}$.

 This lemma is an analog of \cite[Lemmas 10.1.2 and 10.1.3]{ZaiGia}.

\begin{lemma}\label{LemmaOrdinaryAndHopf}
Let $A$ be an associative algebra with a generalized $H$-action
over any field~$F$ and let~$H$ be an associative algebra with $1$. Then
$$c_n(A) \leqslant c^{H}_n(A)
  \leqslant (\dim H)^n c_n(A) \text{ for all } n \in \mathbb N.$$
\end{lemma}
\begin{remark}
If $A$ is an algebra graded by a finite group $G$, we apply Lemma~\ref{LemmaGradAction} and obtain the well known bounds for graded codimensions:  
 $c_n(A) \leqslant c^{\mathrm{gr}}_n(A)
  \leqslant |G|^n c_n(A)$ for all  $n \in \mathbb N$.
\end{remark}
\begin{proof}[Proof of Lemma~\ref{LemmaOrdinaryAndHopf}]
As in Lemma~\ref{LemmaCodimDim}, we consider polynomials as $n$-linear maps from $A$ to $A$ and identify  $\frac{P_n}{P_n \cap \Id(A)}$ and $\frac{P^{H}_n}{P^{H}_n \cap \Id^H(A)}$
with the corresponding subspaces in $\Hom_{F}(A^{{}\otimes n}; A)$.
Then
 $$\frac{P_n}{P_n \cap \Id(A)} \subseteq \frac{P^{H}_n}{P^{H}_n \cap \Id^H(A)}
\subseteq \Hom_{F}(A^{{}\otimes n}; A)$$
and the lower bound follows.

Choose such $f_1, \ldots, f_t \in P_n$ that their images form a basis in 
$\frac{P_n}{P_n \cap \Id(A)}$.  Then for any monomial $x_{\sigma(1)}x_{\sigma(2)}
\ldots x_{\sigma(n)}$, $\sigma \in S_n$, there exist $\alpha_{i,\sigma} \in F$ such that
\begin{equation}\label{EqPnModuloId}x_{\sigma(1)}x_{\sigma(2)}
\ldots x_{\sigma(n)} - \sum_{i=1}^t \alpha_{i,\sigma} f_i(x_1, \ldots, x_n) \in \Id(A).
\end{equation}

Let $(\gamma_j)_{j=1}^m$ be a basis in $H$. Then $x^{\gamma_{i_1}}_{\sigma(1)}x^{\gamma_{i_2}}_{\sigma(2)}
\ldots x^{\gamma_{i_n}}_{\sigma(n)}$, $\sigma \in S_n$, $1 \leqslant i_j \leqslant m$,
form a basis in $P^H_n$. Note that (\ref{EqPnModuloId}) implies
$$x^{\gamma_{i_1}}_{\sigma(1)}x^{\gamma_{i_2}}_{\sigma(2)}
\ldots x^{\gamma_{i_n}}_{\sigma(n)} - \sum_{i=1}^t \alpha_{i,\sigma} f_i(x^{\gamma_{i_1}}_1,
 \ldots, x^{\gamma_{i_n}}_n) \in \Id^H(A).$$
Hence any $H$-polynomial from $P^H_n$ can be expressed modulo $\Id^H(A)$
as a linear combination of $H$-polynomials $f_i(x^{\gamma_{i_1}}_1,
 \ldots, x^{\gamma_{i_n}}_n)$. The number of such polynomials equals $m^n t = (\dim H)^n c_n(A)$
 that finishes the proof.
\end{proof}

\subsection{Generalized Hopf PI-exponent}\label{SubsectionGenHopfPIexp}

Let $A$ be an algebra with a generalized $H$-action.
We call $\PIexp^H(A):=\lim\limits_{n\to\infty}
 \sqrt[n]{c^H_n(A)}$ (if it exists) the \textit{generalized Hopf PI-exponent} of $A$.

Theorem~\ref{TheoremMainH}
  is a particular case of 
\begin{theorem}\label{TheoremMain}
Let $A$ be a finite dimensional non-nilpotent associative algebra
with a generalized $H$-action
over an algebraically closed field $F$ of characteristic $0$.  Here $H$ is a finite dimensional
associative algebra with $1$ acting on $A$ in such a way that the Jacobson radical $J:=J(A)$
is $H$-invariant and $A = B \oplus J$
(direct sum of $H$-submodules) where $B=B_1 \oplus \ldots \oplus B_q$ (direct sum of $H$-invariant ideals),
$B_i$ are $H$-simple semisimple algebras.
 Then there exist constants $C_1, C_2 > 0$, $r_1, r_2 \in \mathbb R$, $d \in \mathbb N$ such that $C_1 n^{r_1} d^n \leqslant c^{H}_n(A) \leqslant C_2 n^{r_2} d^n$ for all $n \in \mathbb N$.
\end{theorem}

Theorem~\ref{TheoremMain} is proved in Sections~\ref{SectionUpper}--\ref{SectionLower}.

Let $$ d(A):= \max(\dim(
B_{i_1} \oplus B_{i_2} \oplus \ldots \oplus
B_{i_r}) \mid B_{i_1}J B_{i_2}J
\ldots JB_{i_r}\ne 0,$$ \begin{equation}\label{EqdofA} 1 \leqslant i_k \leqslant q,
  1 \leqslant k \leqslant r;\
0 \leqslant r \leqslant q )
.\end{equation}

We claim that $\PIexp^{H}(A) = d(A)$ and prove
 Theorem~\ref{TheoremMain} for $d=d(A)$.

Note that since $J$ is an ideal of $A$,
$$d(A)= \max(\dim(
B_{i_1} + B_{i_2} + \ldots +
B_{i_r}) \mid B_{i_1}J B_{i_2}J
\ldots JB_{i_r}\ne 0,$$$$ 1 \leqslant i_k \leqslant q,
  1 \leqslant k \leqslant r;\ r \in \mathbb Z_+)
.$$
 
 \subsection{Proof of Theorem~\ref{TheoremMainG}}\label{SubsectionProofofTheoremMainG}

In order to deduce Theorem~\ref{TheoremMainG} from
Theorem~\ref{TheoremMain}, we need the following generalization of the Wedderburn theorem
obtained by Maschke's argument.
\begin{lemma}\label{LemmaWedderburnG}
Let $B$ be a finite dimensional semisimple algebra over a 
field $F$ of characteristic $0$ with $G$-action by automorphisms and anti-automorphisms.
Then $B=B_1 \oplus B_2 \oplus \ldots \oplus B_q$ (direct sum of ideals)
where $B_i$ are $G$-simple algebras.
\end{lemma}
\begin{proof}
The property of complete reducibility in a finite dimensional case is equivalent
to the splitting property. Here we use this idea.

If $I_1$ is a $G$-invariant ideal, by the original Wedderburn theorem,
there exists an ideal $I_2$ such that $I_1 \oplus I_2 = B$.
Consider the projection $\pi \colon B \to I_1$ along $I_2$.
Then $\pi(ab)=a\pi(b)=\pi(a)b$ for all $a,b \in B$.
Let $\tilde\pi(a) := \frac{1}{|G|}\sum_{g \in G} g \tilde\pi(g^{-1}a)$ for all $g \in G$.
Note that $\pi(ab)=a\pi(b)=\pi(a)b$ and $\pi(ga)=g\pi(a)$ for all $a,b \in B$
and $g \in G$. Hence $\ker \tilde\pi$ is a two-sided $G$-invariant ideal of $B$.
 Moreover $\im \tilde\pi = I_1$ and $\tilde\pi\Bigl|_{I_1}=\id_{I_1}$. Thus $B = I_1 \oplus \ker \tilde\pi$,
and the splitting property is proved. Hence $B=B_1 \oplus B_2 \oplus \ldots \oplus B_q$ (direct sum of ideals)
for some $G$-simple $B_i$.
\end{proof}

\begin{proof}[Proof of Theorem~\ref{TheoremMainG}]
Since codimensions do not change upon an extension of the base field,
we may assume $F$ to be algebraically closed.
Moreover, $J:=J(A)$ is $G$-invariant since the image of a nilpotent ideal is nilpotent.
By $G$-invariant Wedderburn~--- Malcev theorem~\cite[Theorem~1, Remark~1]{Taft},
$A=B\oplus J$ where $B$ is a $G$-invariant maximal semisimple subalgebra.
In Lemma~\ref{LemmaWedderburnG} we have proved the $G$-invariant Wedderburn theorem.
Now we use Example~\ref{ExampleIdFG*} and Theorem~\ref{TheoremMain}.
\end{proof}

\subsection{$S_n$-cocharacters}\label{SubsectionSnCocharacters}

One of the main tools in the investigation of polynomial
identities is provided by the representation theory of symmetric groups.
 The symmetric group $S_n$  acts
 on the space $\frac {P^H_n}{P^H_{n}
  \cap \Id^H(A)}$
  by permuting the variables.
  Irreducible $FS_n$-modules are described by partitions
  $\lambda=(\lambda_1, \ldots, \lambda_s)\vdash n$ and their
  Young diagrams $D_\lambda$.
   The character $\chi^H_n(A)$ of the
  $FS_n$-module $\frac {P^H_n}{P^H_n
   \cap \Id^H(A)}$ is
   called the $n$th
  \textit{cocharacter} of polynomial $H$-identities of $A$.
  We can rewrite it as
  a sum $$\chi^H_n(A)=\sum_{\lambda \vdash n}
   m(A, H, \lambda)\chi(\lambda)$$ of
  irreducible characters $\chi(\lambda)$.
Let  $e_{T_{\lambda}}=a_{T_{\lambda}} b_{T_{\lambda}}$
and
$e^{*}_{T_{\lambda}}=b_{T_{\lambda}} a_{T_{\lambda}}$
where
$a_{T_{\lambda}} = \sum_{\pi \in R_{T_\lambda}} \pi$
and
$b_{T_{\lambda}} = \sum_{\sigma \in C_{T_\lambda}}
 (\sign \sigma) \sigma$,
be Young symmetrizers corresponding to a Young tableau~$T_\lambda$.
Then $M(\lambda) = FS e_{T_\lambda} \cong FS e^{*}_{T_\lambda}$
is an irreducible $FS_n$-module corresponding to
 a partition~$\lambda \vdash n$.
  We refer the reader to~\cite{Bahturin, DrenKurs, ZaiGia}
   for an account
  of $S_n$-representations and their applications to polynomial
  identities.

\medskip

In Section~\ref{SectionUpper}
we prove that if $m(A, H, \lambda) \ne 0$, then
the corresponding Young diagram $D_\lambda$
has at most $d$ long rows. This implies the upper bound.

In Section~\ref{SectionAlt} we
consider $H$-simple algebras $B_0$.
For arbitrary $k\in\mathbb N$,
we construct a $H$-polynomial that is
alternating in $2k$ sets,
 each consisting of $\dim B_0$ variables. This polynomial
 is not an identity of $B_0$.  
 In Section~\ref{SectionLower} we
 glue the alternating polynomials, corresponding to $B_{i_k}$
 from the definition of $d(A)$.
 This allows us to find $\lambda \vdash n$
 with $m(A, H, \lambda)\ne 0$ such that $\dim M(\lambda)$
 has the desired asymptotic behavior, and the lower bound is proved.

\section{Upper bound}\label{SectionUpper}

In
Sections~\ref{SectionUpper}--\ref{SectionLower} we consider
 arbitrary~$A$ and~$H$ satisfying the conditions
of Theorem~\ref{TheoremMain}. In particular, $\mathop\mathrm{char} F = 0$,
the Jacobson radical $J:=J(A)$
is $H$-invariant and $A = B \oplus J$
(direct sum of subspaces) where $B=B_1 \oplus \ldots \oplus B_q$ (direct sum of ideals),
$B_i$ are $H$-simple semisimple algebras.

Since $J$ is the Jacobson radical of $A$,
we have $J^{p} = 0$ for some $p\in \mathbb N$.

\begin{lemma}\label{LemmaUpper}
If $\lambda = (\lambda_1, \ldots, \lambda_s) \vdash n$
and $\sum_{i=d+1}^s \lambda_i \geqslant p$
or $\lambda_{\dim A+1} > 0$, then
$m(A, H, \lambda) = 0$.
\end{lemma}

\begin{proof}
It is sufficient to prove that $e^{*}_{T_\lambda} f \in \Id^H(A)$
for every $f\in P^H_n$ and a Young tableau $T_\lambda$,
 $\lambda \vdash n$, with
$\sum_{i=d+1}^s \lambda_i \geqslant p$
or $\lambda_{\dim A+1} > 0$.

Fix some basis of $A$ that is a union of
bases of $B_i$, $1\leqslant i \leqslant q$, and $J$.
Since polynomials are multilinear, it is
sufficient to substitute only basis elements.
 Note that
$e^{*}_{T_\lambda} = b_{T_\lambda} a_{T_\lambda}$
and $b_{T_\lambda}$ alternates the variables of each column
of $T_\lambda$. Hence if we make a substitution and $
e^{*}_{T_\lambda} f$ does not vanish, then this implies
 that different basis elements
are substituted for the variables of each column.
But if $\lambda_{\dim A+1} > 0$, then the length of the first
 column is greater
than $\dim A$. Therefore,
 $e^{*}_{T_\lambda} f \in \Id^H(A)$.

Consider the case $\sum_{i=d+1}^s \lambda_i \geqslant p$.
 Fix a substitution of basis elements for the variables
 $x_1, \ldots, x_n$. Suppose $e^{*}_{T_\lambda}f$
 does not vanish under this substitution.
Note that if a product of basis elements does not equal
zero, then, by the definition of $d=d(A)$,
we can choose such $1 \leqslant i_1, \ldots, i_r \leqslant q$,
$0 \leqslant r \leqslant q$, that all these
basis elements belong to
$B_{i_1} \oplus \ldots \oplus B_{i_r} \oplus J$
and $$\dim(B_{i_1} \oplus \ldots \oplus B_{i_r})\leqslant d.$$
Since the polynomial $e^{*}_{T_\lambda}f$ is alternating
in the variables of each column, we cannot substitute
more than $d$ basis elements from
$B_{i_1} \oplus \ldots \oplus B_{i_r}$ in each column.
Thus we must substitute at least
$\sum_{i=d+1}^s \lambda_i \geqslant p$
elements from $J$ for the variables of
$e^{*}_{T_\lambda}f$. Hence $e^{*}_{T_\lambda}f$
vanishes since $J^p = 0$. We get a contradiction.
Therefore, $e^{*}_{T_\lambda}f \in \Id^H(A)$.
\end{proof}

Now we can prove
\begin{theorem}\label{TheoremUpper} If $d > 0$, then
there exist constants $C_2 > 0$, $r_2 \in \mathbb R$
such that $c^H_n(A) \leqslant C_2 n^{r_2} d^n$
for all $n \in \mathbb N$. In the case $d=0$,
 the algebra~$A$ is nilpotent.
\end{theorem}
\begin{proof} Suppose $d > 0$.
Lemma~\ref{LemmaUpper} and~\cite[Lemmas~6.2.4, 6.2.5]{ZaiGia}
imply
$$
\sum_{m(A,H, \lambda)\ne 0} \dim M(\lambda) \leqslant C_3 n^{r_3} d^n
$$
for some constants $C_3, r_3 > 0$.
By~\cite[Theorem~13~(b) and the remark after Theorem~14]{BereleHopf},
 there exist constants $C_4 > 0$, $r_4 \in \mathbb N$ such that
 $$\sum_{\lambda \vdash n} m(A,H,\lambda)
\leqslant C_4n^{r_4}$$ for all $n \in \mathbb N$.
This implies the upper bound.

Note that $d(A) \geqslant \dim B_i$ for all $1 \leqslant i \leqslant q$.
Thus $d=d(A)=0$ implies $q=0$ and $A=J$. Therefore, $A$ is nilpotent.
\end{proof}

\section{Alternating polynomials}\label{SectionAlt}

In this section we prove auxiliary propositions needed
 to obtain the lower bound.

 Let $B_0$ be an $H$-simple semisimple algebra over an algebraically closed field $F$
 of characteristic $0$ endowed with a generalized Hopf action of
 a finite dimensional associative algebra $H$ with $1$.
 
Denote by $\varphi \colon B_0 \to \End_F(B_0)$
the left regular representation of $B_0$, i.e.
$\varphi(a)b=ab$, $a,b \in B_0$,
and by $\psi \colon B_0 \to \End_F(B_0)$
the right regular representation of $B_0$, i.e.
$\psi(a)b=ba$, $a,b \in B_0$.
Let $\rho$ be the representation $H \to \End_F(B_0)$.

Note that~(\ref{EqGeneralizedHopf}) implies \begin{equation}\label{EqOpComm1}
\rho(h)\varphi(a)=
 \sum_i \Bigl(\varphi(h'_i a)\rho(h''_i)+\psi(h''''_i a)\rho(h'''_i)\Bigr),
\end{equation}
\begin{equation}\label{EqOpComm2}
\rho(h)\psi(a)=
 \sum_i \Bigl(\psi(h''_i a)\rho(h'_i)+\varphi(h'''_i a)\rho(h''''_i)\Bigr),
\end{equation}
\begin{equation}\label{EqOpComm3}
\varphi(a)\psi(b)=\psi(b)\varphi(a)
\end{equation}
for all $a,b \in B_0$.

\begin{lemma}\label{LemmaForm}
The bilinear form $\tr(\varphi(\cdot)\varphi(\cdot))$
is non-degenerate on $B_0$.
\end{lemma}
\begin{proof}
Recall that $B_0$ is semisimple.
 Thus $$B_0\cong M_{k_1}(F)\oplus M_{k_2}(F)
\oplus \ldots \oplus M_{k_s}(F)$$
for some $k_i \in \mathbb N$ where
 $M_{k_i}(F)$ are algebras of $k_i\times k_i$
matrices.

Fix the bases in $M_{k_i}(F)$ that consist of matrix units
 $e_{\alpha\beta}^{(i)}$.
Note that $$\tr(\varphi(e_{\alpha\beta}^{(i)}))=\left\lbrace
\begin{array}{lll} 0 & \text{ if } & \alpha \ne \beta, \\
 k_i & \text{ if } & \alpha = \beta.
\end{array} \right.$$ Thus $\tr(\varphi(e_{\alpha\beta}^{(i)})
 \varphi(b)) \ne 0$
for some basis element $b$ if and only
if $b=e_{\beta\alpha}^{(i)}$.
Hence the matrix of the bilinear
form is non-degenerate.
\end{proof}

Lemma~\ref{LemmaAlternateFirst} is an
 analog of~\cite[Lemma~1]{GiaSheZai}.

\begin{lemma}\label{LemmaAlternateFirst}
Let $a_1, \ldots, a_\ell$ be a basis of $B_0$.
 For some $T\in\mathbb N$ there exists
 a polynomial $$f=f(x_1, \ldots, x_\ell,
y_1, \ldots, y_\ell, z_1, \ldots, z_T, z) \in P^H_{2\ell+T+1}$$
  alternating in $\lbrace x_1, \ldots, x_\ell \rbrace$
   and in $\lbrace y_1, \ldots, y_\ell \rbrace$
     where $\ell=\dim B_0$, satisfying the following property:
     there exist $ \bar z_1, \ldots, \bar z_T \in B_0$
such that for any $\bar z \in B_0$ we have
$f(a_1, \ldots, a_\ell,
a_1, \ldots, a_\ell, \bar z_1, \ldots, \bar z_T, \bar z) = \bar z$.
\end{lemma}
\begin{proof}
Since $B_0$ is $H$-simple, by the density theorem,
 $\End_F(B_0) \cong M_\ell(F)$ is generated by
 operators from $\rho(H)$,
$\varphi(B_0)$, and $\psi(B_0)$.
By~(\ref{EqOpComm1})--(\ref{EqOpComm3}),
\begin{equation}\label{EqEndB_0}
\End_F(B_0) = \langle \varphi(a)\psi(b)\rho(h)
\mid a,b \in B_0, h \in H \rangle_F.
\end{equation}

Consider Regev's polynomial
$$ f_\ell(x_1, \ldots, x_{\ell^2}; y_1, \ldots, y_{\ell^2})
=\sum_{\substack{\sigma \in S_\ell, \\ \tau \in S_\ell}}
 (\sign(\sigma\tau))
x_{\sigma(1)}\ y_{\tau(1)}\ x_{\sigma(2)}x_{\sigma(3)}x_{\sigma(4)}
\ y_{\tau(2)}y_{\tau(3)}y_{\tau(4)}\ldots
$$
$$ x_{\sigma\left(\ell^2-2\ell+2\right)}\ldots
 x_{\sigma\left(\ell^2\right)}
\ y_{\tau\left(\ell^2-2\ell+2\right)}\ldots
 y_{\tau\left(\ell^2\right)}.
$$ This is a central polynomial~\cite[Theorem~5.7.4]{ZaiGia}
for $M_\ell(F)$, i.e. $f_\ell$ is not a polynomial
 identity for $M_\ell(F)$
and its values belong to the center of $M_\ell(F)$.

Note that $\ker \varphi =0 $ since $B_0$ is semisimple and, therefore, has a unit element. 
Thus $\varphi$ is a monomorphism.
By~(\ref{EqEndB_0}),
$$\varphi(a_1), \ldots, \varphi(a_\ell),
 \quad \varphi(a_{i_1})
\psi(a_{k_1})\rho(h_1),\quad
\ldots, \quad  \varphi(a_{i_s})
\psi(a_{k_s})\rho(h_s)$$ form a basis of $\End_F(B_0)$
for appropriate $i_t,k_t \in \lbrace 1,2, \ldots, \ell \rbrace$,
 $h_t\in H$.
 Now  we replace $x_i$ in $f_\ell$ with $\varphi(x_i)$,
 and $y_i$ with $\varphi(y_i)$ for $1 \leqslant i \leqslant \ell$. Moreover
 we replace  $x_{\ell+j}$
 with $\varphi(z_j)
\psi(u_j)\rho(h_j)$,
and $y_{\ell+j}$
 with $\varphi(v_j)
\psi(w_{j})\rho(h_j)$ for $1 \leqslant j \leqslant s$.
Here $x_i$, $y_i$, $z_i$, $u_i$, $v_i$, $w_i$
are variables with values in $B_0$.
Denote the function obtained by $\tilde f_\ell$.
If we substitute $z_t=v_t=a_{i_t}$,
$u_t=w_t=a_{k_t}$,
$1 \leqslant t \leqslant s$;
$x_i=y_i = a_i$, $1 \leqslant i \leqslant \ell$,
then $\tilde f_\ell$ becomes a scalar operator $(\mu \id_{B_0})$ on $B_0$,
$\mu\in F$, $\mu \ne 0$.
Now we introduce a new variable $z$.
Using~(\ref{EqOpComm1})--(\ref{EqOpComm3}), we move in
$\tilde f_\ell \cdot z$ all $\rho(h_i)$
to the right, and rewrite all $\varphi(\ldots)$
and $\psi(\ldots)$ by the definition.
Then $f := \mu^{-1} \tilde f_\ell \cdot z$ becomes a polynomial from $P^H_n$
for some $n \in \mathbb N$. Let $T := 4s$. Rename $u_i$, $v_i$, $w_i$
to $z_j$ where $s+1 \leqslant j \leqslant T$.
 Then $f$
satisfies all the conditions of the lemma.
\end{proof}

Let $k\ell \leqslant n$ where $k,n \in \mathbb N$ are some numbers.
 Denote by $Q^H_{\ell,k,n} \subseteq P^H_n$
the subspace spanned by all polynomials that are alternating in
$k$ disjoint subsets of variables $\{x^i_1, \ldots, x^i_\ell \}
\subseteq \lbrace x_1, x_2, \ldots, x_n\rbrace$, $1 \leqslant i \leqslant k$.

Theorem~\ref{TheoremAlternateFinal}
 is an analog of~\cite[Theorem~1]{GiaSheZai}.

\begin{theorem}\label{TheoremAlternateFinal}
Let $B_0$ be a $H$-simple semisimple associative algebra over a 
algebraically closed field $F$
of characteristic $0$ endowed with a generalized Hopf action of
 a finite dimensional associative algebra $H$ with $1$.
 Let $a_1, \ldots, a_\ell$ be a basis of $B_0$.
Then there exist $T \in \mathbb Z_+$ and  $\bar z_1, \ldots, \bar z_T \in B_0$ such that
for any $k \in \mathbb N$
there exists $$f=f(x_1^1, \ldots, x_\ell^1; \ldots;
x^{2k}_1, \ldots,  x^{2k}_\ell;\ z_1, \ldots, z_T;\ z) \in Q^H_{\ell, 2k, 2k\ell+T+1}$$
such that for any $\bar z \in B_0$ we have
$f(a_1, \ldots, a_\ell; \ldots;
a_1, \ldots, a_\ell; \bar z_1, \ldots, \bar z_T; \bar z) = \bar z$.
\end{theorem}
\begin{proof}
Let $f_1=f_1(x_1,\ldots, x_\ell,\ y_1,\ldots, y_\ell,
z_1, \ldots, z_T, z)$ be the polynomial from Lemma~\ref{LemmaAlternateFirst}
alternating in $x_1,\ldots, x_\ell$ and in $y_1,\ldots, y_\ell$.
Note that $f_1$ satisfies all the conditions of the theorem for $k=1$.
Thus we may assume that $k > 1$. Note that
$$
f^{(1)}_1(u_1, v_1, x_1, \ldots, x_\ell,\ y_1,\ldots, y_\ell,
z_1, \ldots, z_T, z) :=$$ $$
\sum^\ell_{i=1} f_1(x_1, \ldots, u_1 v_1 x_i,  \ldots, x_\ell,\ y_1,\ldots, y_\ell,
z_1, \ldots, z_T, z)$$
is alternating in $x_1,\ldots, x_\ell$ and in $y_1,\ldots, y_\ell$ too
and $$
f^{(1)}_1(\bar u_1, \bar v_1, \bar x_1, \ldots, \bar x_\ell,\
\bar y_1,\ldots, \bar y_\ell,
\bar z_1, \ldots, \bar z_T, \bar z) =$$
$$
 \tr(\varphi( \bar u_1) \varphi(\bar v_1))
f_1(\bar x_1, \bar x_2, \ldots, \bar x_\ell,\ \bar y_1,\ldots, \bar y_\ell,
\bar z_1, \ldots, \bar z_T, \bar z)
$$
 for any substitution of elements from $B_0$
 since we may assume that $\bar x_1, \ldots, \bar x_\ell$ are different basis elements.

Let $$
f^{(j)}_1(u_1, \ldots, u_j, v_1, \ldots, v_j, x_1, \ldots, x_\ell,\ y_1,\ldots, y_\ell,
z_1, \ldots, z_T, z) :=$$ $$
\sum^\ell_{i=1} f^{(j-1)}_1(u_1, \ldots,  u_{j-1}, v_1, \ldots, v_{j-1},
 x_1, \ldots, u_j v_j x_i,  \ldots, x_\ell,\ y_1,\ldots, y_\ell,
z_1, \ldots, z_T, z),$$
$2 \leqslant j \leqslant \ell$.
Again,
$$
f^{(j)}_1(\bar u_1, \ldots, \bar u_j, \bar v_1, \ldots, \bar v_j, \bar x_1, \ldots, \bar x_\ell,\ \bar y_1,\ldots, \bar y_\ell, \bar z_1, \ldots, \bar z_T, \bar z) =$$
$$ \tr(\varphi( \bar u_1) \varphi(\bar v_1))
 \tr(\varphi( \bar u_2) \varphi(\bar v_2))
 \ldots
 \tr(\varphi( \bar u_j) \varphi(\bar v_j))
 \cdot
$$
\begin{equation}\label{EqKilling}
\cdot
f_1(\bar x_1, \bar x_2, \ldots, \bar x_\ell,\ \bar y_1,\ldots, \bar y_\ell,
\bar z_1, \ldots, \bar z_T, \bar z).
\end{equation}

Note $\det(\tr(\varphi(a_i) \varphi(a_j)))_{i,j=1}^\ell \ne 0$ since the form $\tr(\varphi(\cdot) \varphi(\cdot))$ is non-degenerate
by Lemma~\ref{LemmaForm}.
We define
$$ f_2(u_1, \ldots, u_\ell, v_1, \ldots, v_\ell,
x_1, \ldots, x_\ell, y_1, \ldots, y_\ell, z_1, \ldots, z_T, z) :=
$$ $$\frac{1}{\ell!\det(\tr(\varphi(a_i) \varphi(a_j)))_{i,j=1}^\ell}\sum_{\sigma, \tau \in S_\ell}
\sign(\sigma\tau)
f^{(\ell)}_1(u_{\sigma(1)}, \ldots, u_{\sigma(\ell)}, v_{\tau(1)}, \ldots, v_{\tau(\ell)},$$ $$
 x_1, \ldots, x_\ell,\ y_1,\ldots, y_\ell,
z_1, \ldots, z_T, z).$$
Then $f_2 \in Q^H_{\ell, 4, 4\ell+T+1}$.
Consider a substitution $x_i=y_i=u_i=v_i=a_i$, $1 \leqslant i \leqslant \ell$.
Suppose that the values $z_j=\bar z_j$, $1 \leqslant j \leqslant T$, are chosen
in such a way that $$f_1(a_1, \ldots, a_\ell, a_1, \ldots, a_\ell,
\bar z_1, \ldots, \bar z_T, \bar z) = \bar z \text{\quad for all\quad} \bar z \in B_0.$$
We claim that $$f_2(a_1, \ldots, a_\ell, a_1, \ldots, a_\ell, a_1, \ldots, a_\ell, a_1, \ldots, a_\ell,
\bar z_1, \ldots, \bar z_T, \bar z) = \bar z$$ too.

Indeed,
$$ f_2(a_1, \ldots, a_\ell, a_1, \ldots, a_\ell,
a_1, \ldots, a_\ell, a_1, \ldots, a_\ell, \bar z_1, \ldots, \bar z_T, \bar z) = $$
$$\frac{1}{\ell!\det(\tr(\varphi(a_i) \varphi(a_j)))_{i,j=1}^\ell}\sum_{\sigma, \tau \in S_\ell}
\sign(\sigma\tau)
f^{(\ell)}_1(a_{\sigma(1)}, \ldots, a_{\sigma(\ell)}, a_{\tau(1)}, \ldots, a_{\tau(\ell)}, $$
$$a_1, \ldots, a_\ell,\ a_1,\ldots, a_\ell,
\bar z_1, \ldots, \bar z_T, \bar z).$$ Using~(\ref{EqKilling}), we obtain
$$ f_2(a_1, \ldots, a_\ell, a_1, \ldots, a_\ell,
a_1, \ldots, a_\ell, a_1, \ldots, a_\ell, \bar z_1, \ldots, \bar z_T, \bar z) = $$
$$ \frac{1}{\ell!\det(\tr(\varphi(a_i) \varphi(a_j)))_{i,j=1}^\ell}
\sum_{\sigma, \tau \in S_\ell}
\sign(\sigma\tau) \tr(\varphi(a_{\sigma(1)}) \varphi(a_{\tau(1)}))
  \ldots \tr(\varphi(a_{\sigma(\ell)}) \varphi(a_{\tau(\ell)}))\cdot$$ $$
f_1(a_1, \ldots, a_\ell,\ a_1,\ldots, a_\ell,
\bar z_1, \ldots, \bar z_T, \bar z).
$$
 Note that
$$\sum_{\sigma, \tau \in S_\ell}
\sign(\sigma\tau) \tr(\varphi(a_{\sigma(1)}) \varphi(a_{\tau(1)}))
  \ldots \tr(\varphi(a_{\sigma(\ell)}) \varphi(a_{\tau(\ell)}))
=$$ $$\sum_{\sigma, \tau \in S_\ell}
\sign(\sigma\tau) \tr(\varphi(a_{1}) \varphi(a_{\tau\sigma^{-1}(1)}))
  \ldots \tr(\varphi(a_{\ell}) \varphi(a_{\tau\sigma^{-1}(\ell)}))
\mathrel{\stackrel{(\tau'=\tau\sigma^{-1})}{=}}$$
$$\sum_{\sigma, \tau' \in S_\ell}
\sign(\tau') \tr(\varphi(a_{1}) \varphi(a_{\tau'(1)}))
  \ldots \tr(\varphi(a_{\ell}) \varphi(a_{\tau'(\ell)}))
=$$
$$\ell!\det(\tr(\varphi(a_i) \varphi(a_j)))_{i,j=1}^\ell.$$
Thus $$ f_2(a_1, \ldots, a_\ell, a_1, \ldots, a_\ell,
a_1, \ldots, a_\ell, a_1, \ldots, a_\ell, \bar z_1, \ldots, \bar z_T, \bar z) = \bar z. $$
Note that if $f_1$ is alternating in some of $z_1,\ldots, z_T$,
the polynomial $f_2$
is alternating in those variables too.
Thus if we apply the same procedure to
$f_2$ instead of $f_1$, we obtain $f_3 \in Q^H_{\ell, 6, 6\ell+T+1}$.
Analogously, we define $f_4$ using $f_3$, $f_5$ using $f_4$, etc.
Eventually, we obtain
$f:=f_k \in Q^H_{\ell, 2k, 2k\ell+T+1}$.
\end{proof}

\section{Lower bound}\label{SectionLower}

By the definition of $d=d(A)$, there exist $H$-simple semisimple algebras
$B_{i_1}$, $B_{i_2}$, \ldots, $B_{i_r}$ such that
$$d = \dim(B_{i_1} \oplus B_{i_2} \oplus \ldots \oplus B_{i_r})$$
and
\begin{equation}\label{EqNonZero}
B_{i_1} J B_{i_2} J \ldots J B_{i_r} \ne 0.
\end{equation}
Since $A$ is non-nilpotent, we have $d > 0$.

Our aim is to present a partition $\lambda \vdash n$
with $m(A, H, \lambda)\ne 0$ such that $\dim M(\lambda)$
has the desired asymptotic behavior.
We will glue alternating polynomials constructed
 in Theorem~\ref{TheoremAlternateFinal}
for $B_{i_k}$.

\begin{lemma}\label{LemmaAlt} If $d \ne 0$, then there exist a number $n_0 \in \mathbb N$ such that for every $n\geqslant n_0$
there exist disjoint subsets $X_1$, \ldots, $X_{2k} \subseteq \lbrace x_1, \ldots, x_n
\rbrace$, $k = \left[\frac{n-n_0}{2d}\right]$,
$|X_1| = \ldots = |X_{2k}|=d$ and a polynomial $f \in P^H_n \backslash
\Id^H(A)$ alternating in the variables of each set $X_j$.
\end{lemma}

\begin{proof}
Let $a^{(t)}_{i}$, $1 \leqslant i \leqslant d_t := \dim (B_{i_t})$,
 be a basis in $B_{i_t}$, $1 \leqslant t \leqslant r$.

In virtue of Theorem~\ref{TheoremAlternateFinal},
there exist constants $m_t \in \mathbb Z_+$
such that for any $k$ there exist
 multilinear polynomials $$f_t=f_t(x^{(t, 1)}_1,
 \ldots, x^{(t, 1)}_{d_1};
 \ldots;  x^{(t, 2k)}_1,
 \ldots, x^{(t, 2k)}_{d_1}; z^{(t)}_1, \ldots, z^{(t)}_{m_t}; z_t) \in Q^H_{d_t, 2k, 2k d_t+m_t+1}$$
alternating in the variables from disjoint sets
$X^{(t)}_{\ell}=\lbrace x^{(t, \ell)}_1, x^{(t, \ell)}_2,
\ldots, x^{(t, \ell)}_{d_t} \rbrace$, $1 \leqslant \ell \leqslant 2k$.
There exist $\bar z^{(t)}_\alpha \in B_{i_t}$, $1 \leqslant \alpha \leqslant m_t$,
such that
$$f_t(a^{(t)}_1,
 \ldots, a^{(t)}_{d_t};
 \ldots;  a^{(t)}_1,
 \ldots, a^{(t)}_{d_t}; \bar z^{(t)}_1, \ldots, \bar z^{(t)}_{m_t}; \bar z_t)=\bar z_t$$
 for any $\bar z_t \in B_{i_t}$.

 Let $n_0 = 2r-1+\sum_{i=1}^r m_i$, $k = \left[\frac{n-n_0}{2d}\right]$, $\tilde k =
 \left[\frac{(n-2kd-n_0)-m_1}{2d_1}\right]+1$. We choose $f_t$ for  $B_{i_t}$ and $k$,
 $1 \leqslant t \leqslant r$. In addition, again by Theorem~\ref{TheoremAlternateFinal},
  we take $\tilde f_1$ for $B_{i_1}$ and $\tilde k$.
  Let $$\hat f_0 := f_1\left(x^{(1, 1)}_1,
 \ldots, x^{(1, 1)}_{d_1};
 \ldots;  x^{(1, 2k)}_1,
 \ldots, x^{(1, 2k)}_{d_1}; z^{(1)}_1, \ldots, z^{(1)}_{m_1}; \right.$$ $$\left.
 \tilde f_1(y^{(1)}_1,
 \ldots, y^{(1)}_{d_1};
 \ldots;  y^{(2\tilde k)}_1,
 \ldots, y^{(2\tilde k)}_{d_1}; u_1, \ldots, u_{m_1}; z_1)\right)\cdot $$ $$
 \prod_{t=2}^{r}\left( v_{t-1}\
 f_t(x^{(t, 1)}_1,
 \ldots, x^{(t, 1)}_{d_t};
 \ldots;  x^{(t, 2k)}_1,
 \ldots, x^{(t, 2k)}_{d_t}; z^{(t)}_1, \ldots, z^{(t)}_{m_t}; z_t)\right).$$
 
 Note that without additional manipulations a composition
 of $H$-polynomials is only a multilinear function but not an $H$-polynomial.
 However, using~(\ref{EqGeneralizedHopf}), we can always represent
 such function by an $H$-polynomial. Here we make such manipulations at the very end of the proof.

By~(\ref{EqNonZero}),
$$b_1 j_1 b_2 j_2 \ldots j_{r-1} b_r \ne 0 $$
 for some $b_t \in B_{i_t}$ and $j_t \in J$.
 Denote by $b\in A$
 the value of $\hat f_0$ under the substitution $x^{(t, \alpha)}_{\beta}=a^{(t)}_\beta$,
 $z^{(t)}_{\beta}=\bar z^{(t)}_\beta$, $v_t=j_t$, $z_t=b_t$,
  $y^{(\alpha)}_{\beta}=a^{(1)}_\beta$, $u_{\beta}=\bar z^{(1)}_\beta$.
    Then $$b = b_1 j_1 b_2 j_2 \ldots j_{r-1} b_r \ne 0.$$
    We denote that substitution by $\Xi$.
Let $X_\ell = \bigcup_{t=1}^r X^{(t)}_\ell$ and let $\Alt_\ell$
be the operator of alternation on the set $X_\ell$.
   Denote $\hat f := \Alt_1 \Alt_2 \ldots \Alt_{2k} \hat f_0$.
   Note that the alternations do not change $z_t$,
   and $f_t$ is alternating on each $X^{(t)}_\ell$.
   Hence the value of $\hat f$ under the substitution $\Xi$
   equals $\left((d_1)! (d_2)! \ldots (d_{r})!\right)^{2k}\ b \ne 0$
   since $B_{i_1} \oplus \ldots \oplus B_{i_r}$ is a direct sum
   of ideals and if the alternation puts a variable from
   $X^{(t)}_\ell$ on the place of a variable from $X^{(t')}_\ell$
   for $t \ne t'$, the corresponding $a^{(t)}_\beta$ annihilates $b_{t'}$.

   Note that $\tilde f_1$ is a linear combination
   of multilinear monomials $W$, and one of the terms
   $$\Alt_1 \Alt_2 \ldots \Alt_{2k} f_1(x^{(1, 1)}_1,
 \ldots, x^{(1, 1)}_{d_1};
 \ldots;  x^{(1, 2k)}_1,
 \ldots, x^{(1, 2k)}_{d_1}; z^{(1)}_1, \ldots, z^{(1)}_{m_1}; W)\cdot $$ $$
 \prod_{q=2}^{r}\left( v_{q-1}\
 f_t(x^{(t, 1)}_1,
 \ldots, x^{(t, 1)}_{d_t};
 \ldots;  x^{(t, 2k)}_1,
 \ldots, x^{(t, 2k)}_{d_t}; z^{(t)}_1, \ldots, z^{(t)}_{m_t}; z_t)\right)$$
 in $\hat f$ does not vanish under the substitution $\Xi$.
 Moreover, $$\deg \hat f = 2kd + (2\tilde k d_1+m_1) + n_0 > n$$ and $\deg W = \deg \tilde f_1
 = 2\tilde k d_1+m_1+1$.
 Let $W=w_1 w_2 \ldots w_{2\tilde k d_1+m_1+1}$ where $w_i$ are  variables
 from the set $\lbrace y^{(1)}_1,
 \ldots, y^{(1)}_{d_1};
 \ldots;  y^{(2\tilde k)}_1,
 \ldots, y^{(2\tilde k)}_{d_1}; u_1, \ldots, u_{m_1}; z_1\rbrace$
 replaced under the substitution $\Xi$ with $\bar w_i \in B_{i_1}$.
 Let    $$f:=\Alt_1 \Alt_2 \ldots \Alt_{2k} f_1(x^{(1, 1)}_1,
 \ldots, x^{(1, 1)}_{d_1};
 \ldots;  x^{(1, 2k)}_1,
 \ldots, x^{(1, 2k)}_{d_1}; z^{(1)}_1, \ldots, z^{(1)}_{m_1}; $$ $$
 w_1 w_2 \ldots w_{n-2kd-n_0} z)\cdot$$ $$
 \prod_{q=2}^{r}\left( v_{q-1}\
 f_t(x^{(t, 1)}_1,
 \ldots, x^{(t, 1)}_{d_1};
 \ldots;  x^{(t, 2k)}_1,
 \ldots, x^{(t, 2k)}_{d_1}; z^{(t)}_1, \ldots, z^{(t)}_{m_t}; z_t)\right)$$
where $z$ is an additional variable.
Then using~(\ref{EqGeneralizedHopf}), we may assume $f \in P_n^H$.
Note that $f$ is alternating in $X_\ell$, $1 \leqslant \ell \leqslant 2k$,
and does not vanish under the substitution $\Xi$
with $z = \bar w_{n-2kd-n_0+1} \ldots \bar w_{2\tilde k d_1+m_1+1}$.
 Thus $f$ satisfies all the conditions of the lemma.
\end{proof}

\begin{lemma}\label{LemmaCochar} Let
 $k, n_0$ be the numbers from
Lemma~\ref{LemmaAlt}.   Then for every $n \geqslant n_0$ there exists
a partition $\lambda = (\lambda_1, \ldots, \lambda_s) \vdash n$,
$\lambda_i > 2k-p$ for every $1 \leqslant i \leqslant d$,
with $m(A, H, \lambda) \ne 0$.
\end{lemma}
\begin{proof}
Consider the polynomial $f$ from Lemma~\ref{LemmaAlt}.
It is sufficient to prove that $e^*_{T_\lambda} f \notin \Id^H(A)$
for some tableau $T_\lambda$ of the desired shape $\lambda$.
It is known that $$FS_n = \bigoplus_{\lambda,T_\lambda} FS_n e^{*}_{T_\lambda}$$ where the summation
runs over the set of all standard tableaux $T_\lambda$,
$\lambda \vdash n$. Thus $FS_n f = \sum_{\lambda,T_\lambda} FS_n e^{*}_{T_\lambda}f
\not\subseteq \Id^H(A)$ and $e^{*}_{T_\lambda} f \notin \Id^H(A)$ for some $\lambda \vdash n$.
We claim that $\lambda$ is of the desired shape.
It is sufficient to prove that
$\lambda_d > 2k-p$, since
$\lambda_i \geqslant \lambda_d$ for every $1 \leqslant i \leqslant d$.
Each row of $T_\lambda$ includes numbers
of no more than one variable from each $X_i$,
since $e^{*}_{T_\lambda} = b_{T_\lambda} a_{T_\lambda}$
and $a_{T_\lambda}$ is symmetrizing the variables of each row.
Thus $\sum_{i=1}^{d-1} \lambda_i \leqslant 2k(d-1) + (n-2kd) = n-2k$.
In virtue of Lemma~\ref{LemmaUpper},
$\sum_{i=1}^d \lambda_i > n-p$. Therefore
$\lambda_d > 2k-p$.
\end{proof}

\begin{proof}[Proof of Theorem~\ref{TheoremMain}]
The Young diagram~$D_\lambda$ from Lemma~\ref{LemmaCochar} contains
the rectangular subdiagram~$D_\mu$, $\mu=(\underbrace{2k-p, \ldots, 2k-p}_d)$.
The branching rule for $S_n$ implies that if we consider a restriction of
$S_n$-action on $M(\lambda)$ to $S_{n-1}$, then
$M(\lambda)$ becomes the direct sum of all non-isomorphic
$FS_{n-1}$-modules $M(\nu)$, $\nu \vdash (n-1)$, where each $D_\nu$ is obtained
from $D_\lambda$ by deleting one box. In particular,
$\dim M(\nu) \leqslant \dim M(\lambda)$.
Applying the rule $(n-d(2k-p))$ times, we obtain $\dim M(\mu) \leqslant \dim M(\lambda)$.
By the hook formula, $$\dim M(\mu) = \frac{(d(2k-p))!}{\prod_{i,j} h_{ij}}$$
where $h_{ij}$ is the length of the hook with edge in $(i, j)$.
By Stirling formula,
$$c_n^H(A)\geqslant \dim M(\lambda) \geqslant \dim M(\mu) \geqslant \frac{(d(2k-p))!}{((2k-p+d)!)^d}
\sim $$ $$\frac{
\sqrt{2\pi d(2k-p)} \left(\frac{d(2k-p)}{e}\right)^{d(2k-p)}
}
{
\left(\sqrt{2\pi (2k-p+d)}
\left(\frac{2k-p+d}{e}\right)^{2k-p+d}\right)^d
} \sim C_5 k^{r_5} d^{2kd}$$
for some constants $C_5 > 0$, $r_5 \in \mathbb Q$,
as $k \to \infty$.
Since $k = \left[\frac{n-n_0}{2d}\right]$,
this gives the lower bound.
The upper bound has been proved in Theorem~\ref{TheoremUpper}.
\end{proof}

\section{Examples and applications}\label{SectionExamples}

 Now we apply formula~(\ref{EqdofA}) to calculate the (generalized) Hopf PI-exponent for several important examples. In all of them except the one of Sweedler's algebra we assume $F$
 to be an algebraically closed field of characteristic~$0$.
 
 \subsection{Sums of $H$-simple algebras}

  \begin{example}\label{ExampleHSemiSimple}
 Let $A=B_1 \oplus B_2 \oplus \ldots \oplus B_q$
  be an algebra with a generalized $H$-action,
   where $B_i$ are finite dimensional $H$-simple semisimple
  algebras and $H$ is a finite dimensional associative algebra with~$1$. Let $d := \max_{1 \leqslant k
  \leqslant q} \dim B_k$. Then there exist $C_1, C_2 > 0$, $r_1, r_2 \in \mathbb R$
  such that $$C_1 n^{r_1} d^n \leqslant c_n^{H}(A)
  \leqslant C_2 n^{r_2} d^n \text{ for all } n\in\mathbb N.$$
 \end{example}
 \begin{proof}
 This follows immediately from Theorem~\ref{TheoremMain} and~(\ref{EqdofA}).
 \end{proof}

  \begin{example}\label{ExampleGrSemiSimple}
 Let $A=B_1 \oplus B_2 \oplus \ldots \oplus B_q$
  be a semisimple algebra graded by a finite group, where $B_i$ are finite dimensional graded simple
  algebras. Let $d := \max_{1 \leqslant k
  \leqslant q} \dim B_k$. Then there exist $C_1, C_2 > 0$, $r_1, r_2 \in \mathbb R$
  such that $C_1 n^{r_1} d^n \leqslant c_n^{\mathrm{gr}}(A)
  \leqslant C_2 n^{r_2} d^n$ for all $n\in\mathbb N$.
 \end{example}
 \begin{proof}
 This follows immediately from Lemma~\ref{LemmaGradAction} and Example~\ref{ExampleHSemiSimple}.
 \end{proof}

  \begin{example}\label{ExampleSemiSimpleG}
 Let $A=B_1 \oplus B_2 \oplus \ldots \oplus B_q$
  be a semisimple $G$-algebra where $B_i$ are finite dimensional $G$-simple
  algebras and $G$ is a finite group. Let $d := \max_{1 \leqslant k
  \leqslant q} \dim B_k$. Then there exist $C_1, C_2 > 0$, $r_1, r_2 \in \mathbb R$
  such that $C_1 n^{r_1} d^n \leqslant c_n^G(A)
  \leqslant C_2 n^{r_2} d^n$ for all $n\in\mathbb N$.
 \end{example}
 \begin{proof}
 This follows immediately from Theorem~\ref{TheoremMainG} and~(\ref{EqdofA}).
 \end{proof}

\subsection{Examples of algebras graded by non-Abelian groups}

\begin{example}\label{Example2M2S3}
Let $G=S_3$ and $A=M_2(F)\oplus M_2(F)$. Consider the following $G$-grading
on~$A$: $$A^{(e)} = \left\lbrace\left(\begin{array}{rr}
\alpha & 0 \\
 0 & \beta 
\end{array} \right)\right\rbrace \oplus  \left\lbrace\left(\begin{array}{rr}
\gamma & 0 \\
 0 & \mu 
\end{array} \right)\right\rbrace,$$ $$A^{\bigl((12)\bigr)} = \left\lbrace\left(\begin{array}{rr}
0 & \alpha \\
 \beta & 0  
\end{array} \right)\right\rbrace \oplus  0,\qquad A^{\bigl((23)\bigr)} = 0 \oplus \left\lbrace\left(\begin{array}{rr}
0 & \alpha \\
 \beta & 0  
\end{array} \right)\right\rbrace,$$ the other components are zero.
Then there exist $C_1, C_2 > 0$, $r_1, r_2 \in \mathbb R$
  such that $$C_1 n^{r_1} 4^n \leqslant c_n^{\mathrm{gr}}(A)
  \leqslant C_2 n^{r_2} 4^n \text{ for all } n \in\mathbb N.$$
\end{example}
\begin{proof}
Note that both copies of $M_2(F)$ are simple graded ideals of $A$. Now we apply Example~\ref{ExampleGrSemiSimple}.
\end{proof}

\begin{example}\label{ExampleGroupAlgebra}
Let $A=FG$ where $G$ is a finite group. Consider the natural $G$-grading $A=\bigoplus_{g\in G} A^{(g)}$
where $A^{(g)}=Fg$. Then there exist $C_1, C_2 > 0$, $r_1, r_2 \in \mathbb R$
  such that $$C_1 n^{r_1} |G|^n \leqslant c_n^{\mathrm{gr}}(A)
  \leqslant C_2 n^{r_2} |G|^n \text{ for all } n \in\mathbb N.$$
\end{example}
\begin{proof}
We notice that $FG$ is a graded simple algebra and apply Example~\ref{ExampleGrSemiSimple}.
\end{proof}

\begin{remark}
In~\cite[Corollary~3.4]{AljadeffKanelBelov}, E.~Aljadeff and A.\,Ya.~Kanel-Belov showed that
$$ |G|^n \leqslant c_n^{\mathrm{gr}}(FG) \leqslant |G'|\cdot |G|^n \text{ for all } n\in\mathbb N$$
where $G'$ is the commutator subgroup of $G$. In addition,
they proved that $\lim\limits_{n\to\infty}\left(\frac{c_n^{\mathrm{gr}}(FG)}{|G'|\cdot |G|^n}\right) = 1$. 
\end{remark}

\subsection{Examples of algebras with an action of a non-Abelian group}

  \begin{example}\label{ExampleFields}
  Let $A = F e_1 \oplus \ldots \oplus F{e_m}$ (direct sum of ideals)
  where $e_i^2=e_i$, $m \in \mathbb N$. Suppose $G \subseteq S_m$ acts on $A$
  by the formula $\sigma e_i := e_{\sigma(i)}$, $\sigma \in G$.
  Let $\left\lbrace 1,2, \ldots, m \right\rbrace = \coprod_{i=1}^q O_i$
  where $O_i$ are orbits of the $G$-action on $\left\lbrace 1,2, \ldots, m \right\rbrace$.
  Let $d := \max_{1 \leqslant i
  \leqslant q} |O_i|$. Then there exist $C_1, C_2 > 0$, $r_1, r_2 \in \mathbb R$
  such that $C_1 n^{r_1} d^n \leqslant c_n^G(A)
  \leqslant C_2 n^{r_2} d^n$ for all $n \in\mathbb N$.
 \end{example}
 \begin{proof}
 Note that $A=B_1 \oplus \ldots \oplus B_q$
 where $B_i := \langle e_j \mid j \in O_i \rangle_F$ are $G$-invariant ideals.
 We claim that $B_i$ is $G$-simple for any $1 \leqslant i \leqslant q$.
 Indeed, if $I$ is a nontrivial $G$-invariant ideal of $B_i$,
 there exists $a = \sum_{j \in O_i} \alpha_j e_j \in I$ where $\alpha_j \in F$
 and $\alpha_k \ne 0$ for some $k\in O_i$.
 Thus $e_k = \frac{1}{\alpha_k} e_k a \in I$. Moreover,
 for any $j\in O_i$ there exists $\sigma \in G$ such that  $e_j = \sigma e_{k}$.
  Hence $I=B_i$ and $B_i$ is $G$-simple.

 By Example~\ref{ExampleSemiSimpleG}, $\PIexp^G(A)=\max_{1 \leqslant i
  \leqslant q} \dim B_i = \max_{1 \leqslant i
  \leqslant q} |O_i|$.
 \end{proof}

In Example~\ref{ExampleMatrix} the group may act by anti-automorphisms too.

  \begin{example}\label{ExampleMatrix}
  Let $A = A_1 \oplus \ldots \oplus A_m$ (direct sum of ideals),
   $A_i \cong M_k(F)$, $1\leqslant i \leqslant m$, and $k,m \in \mathbb N$.
   The group $\Aut^*(M_k(F)) \times S_m$
   acts on $A$ in the following way: if $(\varphi,\sigma)\in \Aut^*(M_k(F)) \times S_m$
   and $(a_1, \ldots, a_m) \in A$, then
   $$(\varphi, \sigma) \cdot (a_1, \ldots, a_m)
   := (a^{\varphi}_{\sigma^{-1}(1)}, \ldots, a^{\varphi}_{\sigma^{-1}(m)}).$$
   Suppose $G \subseteq \Aut^*(M_k(F)) \times S_m$ is a subgroup.
   Denote by $\pi : \Aut^*(M_k(F)) \times S_m \to S_m$ the natural
   projection on the second component.
  Let $\left\lbrace 1,2, \ldots, m \right\rbrace = \coprod_{i=1}^q O_i$
  where $O_i$ are orbits of the $\pi(G)$-action on $\left\lbrace 1,2, \ldots, m \right\rbrace$.
  Let $d := k^2 \max_{1 \leqslant i
  \leqslant q} |O_i|$. Then there exist $C_1, C_2 > 0$, $r_1, r_2 \in \mathbb R$
  such that $C_1 n^{r_1} d^n \leqslant c_n^G(A)
  \leqslant C_2 n^{r_2} d^n$ for all $n \in\mathbb N$.
 \end{example}
 \begin{proof}
 Note that $A=B_1 \oplus \ldots \oplus B_q$
 where $B_i := \bigoplus_{j \in O_i} A_j$ are $G$-invariant ideals.
 We claim that $B_i$ is $G$-simple for any $1 \leqslant i \leqslant q$.
 Indeed, if $I$ is a nontrivial $G$-invariant ideal of $B_i$,
 there exists $a = \sum_{j \in O_i} a_j \in I$ where $a_j \in A_j$
 and $a_\ell \ne 0$ for some $\ell \in O_i$. Denote by $e_\ell$
 the identity matrix of $A_\ell$. Then $e_\ell a = a_\ell \in I$
 and $I \cap A_\ell \ne 0$. Since $A_\ell$ is simple,
 $I \cap A_\ell = A_\ell$. Note that for any $j\in O_i$
 there exists $g\in G$ that $A_j = A_\ell^g$.
  Hence $I=B_i$ and $B_i$ is $G$-simple.

 By Example~\ref{ExampleSemiSimpleG}, $\PIexp^G(A)=\max_{1 \leqslant i
  \leqslant q} \dim B_i = k^2 \max_{1 \leqslant i
  \leqslant q} |O_i|$.
 \end{proof}

In Example~\ref{ExampleUT} the algebra is not semisimple.

  \begin{example}\label{ExampleUT}
  Let $A = A_1 \oplus \ldots \oplus A_m$ (direct sum of ideals) where
   $A_i \cong \UT_k(F)$, $1\leqslant i \leqslant m$;  $k,m \in \mathbb N$;
   and $\UT_k(F)$ is the associative algebra of $k\times k$
   upper-triangular matrices.
   Suppose $G \subseteq S_m$ acts on $A$ in the following way: if $\sigma\in G$
   and $(a_1, \ldots, a_m) \in A$, then
   $$\sigma \cdot (a_1, \ldots, a_m)
   := (a_{\sigma^{-1}(1)}, \ldots, a_{\sigma^{-1}(m)}).$$
   Let $\left\lbrace 1,2, \ldots, m \right\rbrace = \coprod_{i=1}^s O_i$
  where $O_i$ are orbits of the $G$-action on $\left\lbrace 1,2, \ldots, m \right\rbrace$.
  Let $d := k \cdot\max_{1 \leqslant i
  \leqslant s} |O_i|$. Then there exist $C_1, C_2 > 0$, $r_1, r_2 \in \mathbb R$
  such that $$C_1 n^{r_1} d^n \leqslant c_n^G(A)
  \leqslant C_2 n^{r_2} d^n
   \text{ for all } n \in\mathbb N.$$
 \end{example}
 \begin{proof} Denote by $e_{ij}^{(t)}$, $1\leqslant i \leqslant j \leqslant k$,
 the matrix units of $A_t$, $1 \leqslant t \leqslant m$.
 Then $\sigma e_{ij}^{(t)} = e_{ij}^{(\sigma(t))}$
 Note that
\begin{equation}\label{EqWedderburnUT}
 A=\left( \bigoplus_{i=1}^s \bigoplus_{j=1}^k B_{ij}\right) \oplus J
 \end{equation}
 where $B_{ij} := \langle e_{jj}^{(t)}  \mid t \in O_i \rangle_F$ are $G$-invariant subalgebras
 and $$J:=\langle e_{ij}^{(t)}  \mid 1\leqslant i < j \leqslant k,
 1 \leqslant t \leqslant m \rangle_F$$ is a $G$-invariant nilpotent ideal.
 We claim that $B_{ij}$ is $G$-simple for any $1 \leqslant i \leqslant s$
 and $1 \leqslant j \leqslant k$.
 Indeed, if $I$ is a nontrivial $G$-invariant ideal of $B_{ij}$,
 there exists $a = \sum_{t \in O_i} \alpha_t e_{jj}^{(t)} \in I$ where $\alpha_t \in F$
 and $\alpha_\ell \ne 0$ for some $\ell \in O_i$.
  Then $e^{(\ell)}_{jj} = \frac{1}{\alpha_\ell} e^{(\ell)}_{jj} a  \in I$.
   Note that for any $t\in O_i$
 there exists $\sigma \in G$ that $e_{jj}^{(t)} = \left(e^{(\ell)}_{jj}
 \right)^\sigma$.
  Hence $I=B_i$ and $B_i$ is $G$-simple. Thus~(\ref{EqWedderburnUT})
  is a $G$-invariant Wedderburn~--- Malcev decomposition of $A$.

  Let  $1 \leqslant i \leqslant s$, $t \in O_i$. Note that
  $$e_{11}^{(t)} e_{12}^{(t)} e_{22}^{(t)} e_{23}^{(t)} \ldots e_{kk}^{(t)} = e_{1k}^{(t)} \ne 0.$$
  Thus $$B_{i1}JB_{i2}J\ldots J B_{ik} \ne 0$$
  and, by Theorem~\ref{TheoremMainG} and~(\ref{EqdofA}), we have
\begin{equation}\label{EqExUTLower}
  \PIexp^G(A)\geqslant \dim(B_{i1} \oplus \ldots \oplus B_{ik})
  = k |O_i|.\end{equation}

  Suppose $$B_{i_1 j_1}JB_{i_2 j_2}J\ldots J B_{i_r j_r} \ne 0$$
  for some $1 \leqslant i_\ell \leqslant s$, $1 \leqslant j_\ell \leqslant k$.
  Then we can choose such $e^{(q_\ell)}_{j_\ell j_\ell} \in B_{i_\ell j_\ell}$,
  $q_\ell \in O_{i_\ell}$, $1 \leqslant j_\ell \leqslant k$,
  and $e^{(q'_\ell)}_{i'_\ell j'_\ell} \in J$, $1 \leqslant i'_\ell < j'_\ell
  \leqslant k$, that
  $$e^{(q_1)}_{j_1 j_1} e^{(q'_1)}_{i'_1 j'_1} e^{(q_2)}_{j_2 j_2} e^{(q'_2)}_{i'_2 j'_2}
  \ldots e^{(q'_{r-1})}_{i'_{r-1} j'_{r-1}} e^{(q_r)}_{j_r j_r} \ne 0.$$
  Thus $q_1 = q'_1 = q_2 = q'_2=\ldots =q'_{r-1}= q_r$ and
  $j_\ell = i'_\ell$, $j'_{\ell-1} = j_\ell$.
  Hence $i_1=\ldots = i_r$, $r \leqslant k$, and
  $ \dim(B_{i_1 j_1} \oplus \ldots \oplus B_{i_r j_r})
  \leqslant k |O_{i_1}|$.
  Therefore, $\PIexp^G(A) \leqslant k\cdot \max_{1 \leqslant i
  \leqslant s} |O_i|$. The lower bound was obtained in~(\ref{EqExUTLower}).
 \end{proof}

\subsection{Sweedler's algebra with the action of its dual}\label{SubsectionSweedler}

Let $H=\langle 1, c, b, cb \rangle_F$ be the 4-dimensional Sweedler's Hopf algebra.
Here $c^2=1$, $b^2=0$, $bc=-cb$, $\Delta(c)=c\otimes c$, $\Delta(b)=c\otimes b + b \otimes 1$,
$\varepsilon(c)=1$, $\varepsilon(b)=0$, $S(c)=c$, $S(b)=-cb$. Then $H$ is a right $H$-comodule
with $\rho=\Delta$. Therefore $H$ is a left $H^*$-module with the action defined by $g h = g(h_{(2)})h_{(1)}$, $h\in H$, $g\in H^*$. In this section we prove

\begin{theorem}\label{TheoremSweedlersAlgebra} Let $F$ be a field of characteristic $0$.
There exist $C > 0$ and $r \in \mathbb R$
such that 
$C n^{r} 4^n \leqslant c^{H^*}_n(H) \leqslant 4^{n+1}$ for all $n\in \mathbb N$.
\end{theorem}

We may assume $F$ to be algebraically closed since the codimensions do not change upon an extension of the base field and $H \otimes_F K$ is again the Sweedler's algebra for any extension $K \supset F$.

 We choose the basis $g_1, g_c, g_b, g_{cb}$ of $H^*$ dual to $1,c,b,cb$ of $H$.
Note that $\Xi \colon H \to H^*$, where $\Xi(1)=g_1 + g_c$, $\Xi(c)=g_1 - g_c$, $\Xi(b)=g_{cb}-g_b$,
$\Xi(cb)=g_{cb}+g_b$, is an isomorphism of Hopf algebras, i.e., in fact, $H$ is acting on itself.
However, for us it is convenient to work with $H^*$-action.
 Put $g_x y$, where $x,y \in \lbrace 1, c, b, cb \rbrace$, in a table:
$$\begin{array}{|c|c|c|c|c|}
\hline
     & 1 & c & b & cb \\
\hline
 g_1 & 1 & 0 & b & 0 \\ 
\hline
 g_c & 0 & c & 0 & cb \\
\hline
 g_b & 0 & 0 & c & 0 \\ 
\hline
 g_{cb} & 0 & 0 & 0 & 1 \\
 \hline
\end{array}.$$

Note that $H$ is an $H^{*}$-simple algebra.
Indeed, suppose $I$ is a nonzero $H^*$-invariant ideal of $H$. Let $a = \alpha 1 + \beta\, c + \gamma\, b + \mu\, cb \in I$,
$\alpha,\beta,\gamma,\mu \in F$. Then $g_b a = \gamma c$ and $g_{cb} a = \mu \, 1$.
Thus if at least one of $\gamma$ and $\mu$ is nonzero, we have either $1 \in I$ or $c \in I$.
In this case $I=H$. If $a = \alpha \, 1 + \beta c$, then $g_1 a = \alpha \, 1$
and $g_c a = \beta \,c$. Again, we obtain $I=H$.

Unfortunately, in the proof of Theorem~\ref{TheoremSweedlersAlgebra} we cannot use Theorem~\ref{TheoremMainH}  since the Jacobson radical of $H$, that equals $\langle b, cb\rangle_F$,
is not $H^*$-invariant. Furthermore, despite the fact that $H$ is an $H^{*}$-simple algebra,
we cannot use the arguments from Section~\ref{SectionAlt} since the trace form on $H$ is degenerate.
However, we can prove Theorem~\ref{TheoremSweedlersAlgebra} directly.

First, we obtain the analog of Lemma~\ref{LemmaAlt}.

\begin{lemma}\label{LemmaSweedlersAlt}
 For every $n\in \mathbb N$
there exist disjoint subsets $X_1$, \ldots, $X_{k} \subseteq \lbrace x_1, \ldots, x_n
\rbrace$, 
$|X_1| = \ldots = |X_k|=4$, $k=\left[\frac{n}{4}\right]$, and a polynomial $f \in P^{H^*}_n \backslash
\Id^{H^*}(H)$ alternating in the variables of each set $X_j$.
\end{lemma}
\begin{proof}
Consider $$f_1 = \sum_{\sigma \in S_4} (\sign \sigma) x^{g_1}_{\sigma(1)}
x^{g_c}_{\sigma(2)} x^{g_b}_{\sigma(3)} x^{g_{cb}}_{\sigma(4)}.$$
Let $x_1 = 1$, $x_2 = c$, $x_3 = b$, $x_4 = cb$. Then only the term
that corresponds to $\sigma = e$ does not vanish. Moreover, $f_1(1,c,b,cb)=1$.

Now we take $$f = \left( \prod_{j=1}^{k} f_1(x_{4j-3}, x_{4j-2}, x_{4j-1}, x_{4j})
\right) x_{4k+1} x_{4k+2} \ldots x_n.$$ Then $f \notin \Id^{H^*}(H)$
since $H$ has $1$. Moreover, $f$ is alternating in the variables
of each set $X_j = \lbrace x_{4j-3}, x_{4j-2}, x_{4j-1}, x_{4j}\rbrace$.
\end{proof}

\begin{proof}[Proof of Theorem~\ref{TheoremSweedlersAlgebra}]
The upper bound is a consequence of Lemma~\ref{LemmaCodimDim}.
In order to obtain the lower bound it is sufficient
to repeat the proof of Lemma~\ref{LemmaCochar} and Theorem~\ref{TheoremMain}
(see the end of Section~\ref{SectionLower})
using Lemma~\ref{LemmaSweedlersAlt} instead of Lemma~\ref{LemmaAlt} for $n_0=0$, $p=1$, $d=4$.
\end{proof}

\section*{Acknowledgements}

I am grateful to Yuri Bahturin who suggested that I study polynomial $H$-identities.
In addition, I appreciate Mikhail Zaicev and Mikhail Kotchetov for helpful
discussions.

\end{document}